\documentclass[11pt]{amsart}

\usepackage{graphicx, tikz}
\usepackage{xcolor}
\usepackage{amssymb,amsfonts,amsthm,amsmath,calligra, enumerate, stmaryrd}
\usepackage[utf8]{inputenc}
\usepackage{accents}
\usepackage{slashed}
\usepackage{yfonts}
\usepackage{mathrsfs,pifont}
\usepackage{enumitem}
\usepackage[bookmarksnumbered=true, bookmarksopen=true]{hyperref}
\usepackage[margin=1in]{geometry}

\usepackage[all]{xy}
\usepackage[colorinlistoftodos]{todonotes}
\theoremstyle{definition}
\newtheorem{Definition}{Definition}[section]
\newtheorem{Assumption}[Definition]{Assumption}
\newtheorem{Conjecture}[Definition]{Conjecture}

\newtheorem{Remark}[Definition]{Remark}
\newtheorem{Example}[Definition]{Example}
\numberwithin{equation}{section}

\theoremstyle{Theorem}
\newtheorem{Theorem}[Definition]{Theorem}
\newtheorem{Proposition}[Definition]{Proposition}
\newtheorem{Lemma}[Definition]{Lemma}
\newtheorem{Corollary}[Definition]{Corollary}

\def\ben{\begin{eqnarray*}}
\def\een{\end{eqnarray*}}

\newcommand{\la}{\langle}
\newcommand{\ra}{\rangle}

\newcommand{\bC}{\mathbb{C}}

\newcommand{\bI}{\mathbb{I}}

\newcommand{\bP}{\mathbb{P}}
\newcommand{\bQ}{\mathbb{Q}}
\newcommand{\bR}{\mathbb{R}}
\newcommand{\bS}{\mathbb{S}}

\newcommand{\bZ}{\mathbb{Z}}

\newcommand{\bb}{\mathbf{b}}

\newcommand{\be}{\mathbf{e}}
\newcommand{\bg}{\mathbf{g}}
\newcommand{\bh}{\mathbf{h}}

\newcommand{\bs}{\mathbf{s}}

\newcommand{\bv}{\mathbf{v}}

\newcommand{\bx}{\mathbf{x}}
\newcommand{\by}{\mathbf{y}}

\newcommand{\bfF}{\mathbf{F}}
\newcommand{\bfI}{\mathbf{I}}
\newcommand{\bfQ}{\mathbf{Q}}
\newcommand{\bfW}{\mathbf{W}}

\newcommand{\cA}{\mathcal{A}}

\newcommand{\cD}{\mathcal{D}}

\newcommand{\cF}{\mathcal{F}}

\newcommand{\cJ}{\mathcal{J}}

\newcommand{\cM}{\mathcal{M}}
\newcommand{\cN}{\mathcal{N}}

\newcommand{\cQ}{\mathcal{Q}}

\newcommand{\cS}{\mathcal{S}}

\newcommand{\cV}{\mathcal{V}}
\newcommand{\cX}{\mathcal{X}}

\newcommand{\fm}{\mathfrak{m}}

\newcommand{\fX}{\mathfrak{X}}

\newcommand{\ab}{\mathrm{ab}}
\newcommand{\BM}{\mathrm{BM}}
\newcommand{\Gr}{\mathrm{Gr}}
\newcommand{\GW}{\mathrm{GW}}
\newcommand{\loc}{\mathrm{loc}}

\newcommand{\nc}{\mathrm{nc}}
\newcommand{\ns}{\mathrm{ns}}

\newcommand{\poly}{\mathrm{poly}}
\newcommand{\rel}{\mathrm{rel}}

\newcommand{\up}{\mathrm{up}}
\newcommand{\vir}{\mathrm{vir}}

\newcommand{\sB}{\mathsf{B}}
\newcommand{\sG}{\mathsf{G}}
\newcommand{\sK}{\mathsf{K}}

\newcommand{\sN}{\mathsf{N}}

\newcommand{\sT}{\mathsf{T}}

\newcommand{\sW}{\mathsf{W}}
\newcommand{\sw}{\mathsf{w}}
\newcommand{\sX}{\mathsf{X}}

\newcommand{\ann}{\operatorname{ann}}

\newcommand{\chr}{\operatorname{char}}

\newcommand{\Eff}{\operatorname{Eff}}

\newcommand{\Frac}{\operatorname{Frac}}
\newcommand{\ev}{\operatorname{ev}}
\newcommand{\gr}{\operatorname{gr}}

\newcommand{\Hom}{\operatorname{Hom}}
\newcommand{\Id}{\operatorname{Id}}
\newcommand{\im}{\operatorname{im}}

\newcommand{\Lie}{\operatorname{Lie}}

\newcommand{\Pic}{\operatorname{Pic}}
\newcommand{\pr}{\operatorname{pr}}

\newcommand{\pt}{\operatorname{pt}}
\newcommand{\QH}{\operatorname{QH}}
\newcommand{\QM}{\operatorname{QM}}

\newcommand{\sign}{\operatorname{sgn}}

\newcommand{\Spec}{\operatorname{Spec}}

\newcommand{\Tot}{\operatorname{Tot}}

\newcommand{\Trop}{\operatorname{Trop}}

\title{Cluster algebra and quasimap quantum cohomology}
\date{\today}

\author{Yingchun Zhang}
\address{School of Mathematical Sciences, Shanghai Jiao Tong University, China}
\email{yingchun.zhang@sjtu.edu.cn}

\author{Zijun Zhou}
\address{School of Mathematical Sciences, Shanghai Jiao Tong University, China}
\email{zijun.zhou@sjtu.edu.cn}

\keywords{Cluster algebra, quiver variety, quasimap quantum cohomology}

\begin{document}

\allowdisplaybreaks

\begin{abstract}
We apply the abelianization technique to obtain an explicit ring presentation for the quasimap quantum cohomology of  GIT quotients. 
As an application, for quiver varieties associated with oriented-acyclic quivers, we establish a cluster algebra structure on their equivariant quasimap quantum cohomology rings.
\end{abstract}

\maketitle

\tableofcontents

\setlength{\parskip}{1ex}

\section{Introduction} 

\subsection{Motivation and background}

Cluster algebra, introduced by S. Fomin and A. Zelevinsky in their series of seminal works \cite{FZ-1, FZ-2, BFZ-3, FZ-4}, has become a fundamental structure connecting diverse areas of mathematics. These include representation theory, algebraic geometry, combinatorics, integrable systems, and beyond.
Informally speaking, it offers a systematic approach to organizing and analyzing the concept of ``mutations", a phenomenon that takes on distinct interpretations across different facets of the theory.

We will be interested in the interaction between cluster algebra and enumerative algebraic geometry.
There are already several works of this type. 
In the remarkable works \cite{GHKK2018, HK2018}, the idea of mirror symmetry plays a crucial role and results in an enumerative geometric understanding of cluster algebras, especially concerning the Laurent phenomenon and positivity. 
Another example is the work of Nagao \cite{Nagao} on Donaldson–Thomas theory, based on the conjecture of Kontsevich–Soibelman ~\cite{kS2008}. 
In that example, mutations can also be understood as wall-crossings from a physics point of view, where the Donaldson--Thomas theory is interpreted as the BPS counting in 4d $\cN=2$ quiver gauge theories. 

The motivation for this paper stems from the \emph{Seiberg duality} \cite{Hori,HoriTong,Benini2015cluster, gomis2016m2}, which originally appears in 4d $\cN=1$ supersymmetric gauge theories. 
Let $(\bfQ, \bfW, \bv)$ be a quiver with potential, with dimension vector $\bv$.
\emph{Quiver mutation} \cite{FZ-1} produces another quiver with potential $(\bfQ', \bfW', \bv')$. See \cite[\S 2]{HZ} for the rule of quiver mutations.
The theories associated with $(\bfQ, \bfW,\bv)$ and $(\bfQ', \bfW',\bv')$ are considered to be Seiberg dual to each other. 

As a general philosophy in enumerative geometry, the enumerative invariants of dual theories are expected to be closely related, very often in terms of some non-trivial change of variables. 
This leads to the  enumerative conjectures.
Let $\sG$ be the gauge group and $\sN$ be the representation space associated with $(\bfQ, \bfW, \bv)$.
Choosing a stability condition $\theta$, the quiver variety is defined as the GIT quotient $X:=\sN /\!\!/_{\theta}\sG$, which admits a superpotential $\sw: X\rightarrow \mathbb C$. 
 One can similarly define the quiver variety $X'$ of $(\bfQ',\bfW',\bv')$ and the superpotential $\sw': X'\rightarrow \mathbb C$.

\begin{Conjecture}
Let $Z:=\{d\sw=0\}\subset X$ and $Z':=\{d\sw'=0\}$ be the critical loci. 

\begin{itemize}
\setlength{\parskip}{1ex}

\item (Seiberg duality conjecture \cite{RuannonabelGLSM2017,Benini2015cluster}) The \emph{Gromov-Witten invariants} (e.g. GW potential, $J$-function) of $Z$ and $Z'$ are equal up to a change of K\"ahler parameters. The transformation rule of K\"ahler variables behaves like that of $\cX$-cluster variables under quiver mutation. 

\item (Cluster algebra conjecture \cite{Benini2015cluster, HZ})
The \emph{quantum cohomology rings}  of $Z$ and $Z'$ are naturally isomorphic. 
Moreover, there is an injective map of algebras
$$
\psi: \cA \to \QH^\GW (Z) [t], 
$$
where $\cA$ is the cluster algebra associated with the quiver $\bfQ$, $\QH^\GW (Z)$ is the quantum cohomology ring, and $t$ is an extra independent variable. 
In particular, the cluster variables are sent to Chern polynomials $c_t (\cV)$ for certain tautological bundles $\cV$ on $Z$.

\end{itemize}

\end{Conjecture}

One of the main difficulties in studying the conjectures is that the critical locus $Z$ is, in general, highly singular, which makes its Gromov-Witten invariants ill-defined. 
A usual idea to tackle this obstacle is to consider the  Landau--Ginzburg model $(\sG, \sN)$, also referred to as the gauge linear sigma model (GLSM) \cite{FJR}. 
This works well for the Seiberg duality conjecture, where one considers the GLSM invariants, or the \emph{$I$-functions}, rather than the usual GW invariants or $J$-functions. 
See ~\cite{ZmutationA,HZmutationD,PMW} for some recent progress. 

However, for the cluster algebra conjecture, although quantum products may be defined for the GLSM $(\sG, \sN)$ \cite{CZ}, it generally does not have an identity. 
In special cases like type $A$ and $D$, the works \cite{ZmutationA,HZmutationD,HZ} involve a careful study of all the mutation types and the geometry of the critical loci, leading to proofs of the two conjectures mentioned above. 
But as the quiver becomes more complicated, it is increasingly difficult to analyze the concrete geometry and exhaust all mutation types. 
In this paper, we will adopt an alternative viewpoint.

\subsection{Main ideas and results}

The use of GLSM or  $I$-functions makes it natural to consider the \emph{quasimap quantum cohomology} $\QH(-)$, instead of the usual one $\QH^\GW (-)$. 
The notion of $(0+)$-stable quasimaps dates back to \cite{MOP, Toda}, and is fully developed by I. Ciocan-Fontanine and B. Kim \cite{CKM, CK-wall, CK-higher-g, CK-big-I} in the study of Gromov--Witten invariants and mirror symmetry. 
In particular, it gives a CohFT for the $\epsilon = 0+$ stability, which leads to the \emph{quasimap quantum cohomology} (see \cite[\S 2.3.3]{CK-higher-g}). 
This version of quantum cohomology (resp. quantum $K$-theory) also appears in \cite{Oko-lec, PSZ, SZ}. 

Another different viewpoint we take is to consider the \emph{ambient space} $X$, rather than the critical locus $Z$.
Recall that for a GLSM $(\sG, \sN)$, with potential $\bfW: \sN \to \bC$, the $I$-function takes values in the Borel–Moore homology $H^\BM_* (Z)$, or the critical cohomology $H^*_{\mathrm{crit}} (\sN, \bfW)$ \cite{CZ}. 
Neither has a natural fundamental class nor an explicit presentation.
However, a typical property of the $I$-function of $Z$ is that its natural pushforward into the ambient space $H_*^\BM (X) \cong H^*(X)$ coincides with the $I$-function of $X$. 
Therefore, the $I$-function of $X$ is often considered as a replacement for the $I$-function of $Z$. 
Similarly, we may also consider the quasimap quantum cohomology of the ambient space $X$ as a promising alternative\footnote{In this way, we focus specifically on investigating the cluster algebra structure solely on $\QH(X)$, without addressing its relationship with $\QH(X')$ (and in fact, they are non-isomorphic).}.

Let us describe our main results.
Let $X=\sN /\!\!/_{\theta}\sG$ be a GIT quotient satisfying some nice geometric assumptions (see \S \ref{Sec-quasimap}). 
Let $\sK\subset \sG$ be the maximal torus, $\sW$ be the Weyl group, and $\sT$ be a torus acting on $\sN$ which commutes with $\sG$. 
The \emph{$\sT$-equivariant quasimap quantum cohomology ring} $\QH_\sT (X)$ is defined in terms of counting $(0+)$-stable quasimaps into $X$ (see Def. ~\ref{def:quasimapqcoh}). 
Applying the abelianization technique \cite{Mar, ES, Web}, we can obtain an explicit presentation that generalizes ~\cite[~Thm. ~4.1.1]{CKS}. 
For notations, we refer to Thm. ~\ref{Thm-QH}.

\begin{Theorem}[Thm. ~\ref{Thm-QH}] 
\label{Thm-QH-intro}

We have
$$
\QH_\sT (X) \cong H_\sT^* (\pt) \otimes_\bC H_{ \sK }^*(\pt)^\sW [\![Q]\!] / \cJ.
$$
where $\cJ$ is the ideal generated by
\begin{equation*} 
\frac{1}{e} \sum_{w\in \sW} (-1)^{l(w)} w \cdot \Bigg[ g(\xi) \cdot \Big(  \prod_{i: \la \lambda_i, d \ra >0} (u_i + \lambda_i )^{\la \lambda_i, d\ra} -   Q^{\bar d}_\sharp \prod_{i: \la \lambda_i, d \ra < 0} ( u_i + \lambda_i )^{- \la \lambda_i, d\ra} \Big)  \Bigg], 
\end{equation*}
for all $d\in \Eff(X^\ab) \subset \sX_*(\sK)$ and $g(\xi) \in H_{ \sK \times \sT}^*(\pt)$. 
\end{Theorem}

The proof of the theorem relies on the abelianization of $I$-functions \emph{with insertions}, as shown in Prop. \ref{Prop-abelianization-I}, which generalizes \cite{Web}.
Utilizing the explicit presentation, some nice properties of $\QH_\sT (X)$ are also obtained, such as finiteness (Cor. ~\ref{Cor-vGIT}) and invariance under variation of GIT (Cor. ~\ref{Cor-vGIT}), the first of which makes a polynomial version $\QH_\sT^\poly (X)$ well-defined. 

For quiver varieties, more interesting results emerge if we appropriately organize the tautological classes.
Let $\bfQ=(\bfQ_0 = \bfI \sqcup \bfF, \bfQ_1)$ be a quiver (see \S \ref{Sec-quiver}), where $\bfI$ and $\bfF$ consist of the gauge and frozen nodes. 
Choosing a dimension vector $\bv \in \bZ^{\bfQ_0}$ and stability condition $\theta$, we have the gauge group $\sG := \prod_{i\in \bfI} GL (\bv_i)$ ,
the maximal torus $\sK:=\prod_{i\in \bfI} (\mathbb C^*)^{\bv_i}$ embedded in $G$ as diagonal matrices, 
the representation space $\sN := \bigoplus_{e\in \bfQ_1} \Hom (V_{s(e)}, V_{t(e)})$, and the quiver variety $X = \sN /\!\!/_\theta \sG$, with a natural flavor torus action $\sT = \prod_{i\in \bf F}(\bC^*)^{\bv_i}$.

For a node $k\in \bfQ_0$, write $V_k = \sum_{j=1}^{\bv_k} \xi^{(k)}_j$, with $\xi_j^{(k)} \in \sX^*(\sK)$. 
The $Q = 0$ specialization of the presentation in Thm. ~\ref{Thm-QH-intro} gives the ordinary cohomology $H_\sT^*(X) \cong H^*_{\sK\times \sT} (\pt)^\sW / \cJ_{Q=0}$, which coincides with the Kirwan surjection.
In particular, the image of $V_k$ is the tautological bundle $\cV_k \in H_\sT^*(X)$. 
Let $V_- := \bigoplus_{e\in \bfQ_1, \, t(e) = k} V_{s(e)}$ and $V_+ := \bigoplus_{e\in \bfQ_1, \, s(e) = k} V_{t(e)}$, with natural maps $V_- \to V_k$ and $V_k \to V_+$.

Introducing an extra parameter $t$, one can formally consider the Chern polynomial $c_t(V_k)=\prod_{j=1}^{\bv_k}(t+\xi^{(k)}_j)$, and the \emph{truncated Chern quotient} $\delta_t (V_\pm, V_k) := [ c_t (V_\pm) c_t (V_k)^{-1} ]_+$ (see Def. ~\ref{eqn:chernquot}). 
It turns out that these classes produce nice quantum relations.

\begin{Theorem}[Thm. ~\ref{Thm-c_t-rel}]\label{introd:thmexchangerelation}
Suppose $\theta_k>0$, we have
$$
c_t ( V_- )  - \delta_t (V_- , V_k) \cdot c_t (V_k)  =  (-1)^{\bv_- -\bv+1} Q^{(k)} \cdot \left( c_t ( V_+ )
- \delta_t (V_+ ,  V_k)  \cdot c_t (V_k) \right). 
$$
When $\theta_k<0$, there is a similar formula. 
\end{Theorem}

A simple rearrangement yields
\begin{align}\label{eqn:qclusterexchangerelation}
    c_t (V_k) \cdot \left(\delta_t (V_- , V_k)+(-1)^{\bv_- -\bv_k}Q^{(k)}  \cdot\delta_t (V_+ ,  V_k)\right) =    c_t ( V_- )+(-1)^{\bv_- -\bv_k} Q^{(k)} \cdot c_t ( V_+). 
\end{align}
In particular,  since $\theta_k >0$, $V_- \to V_k$ is surjective. 
When $Q\rightarrow 0$, the specialization of $\delta_t (V_- , V_k)+(-1)^{\bv_- -\bv_k}Q^{(k)}  \cdot\delta_t (V_+ ,  V_k)$ in $H_\sT^*(X)$ is exactly the Chern polynomial of the quotient bundle $c_t (\cV_-/\cV_k)$.

The relation  \eqref{eqn:qclusterexchangerelation} resembles the cluster exchange relation,  leading to our main theorem, which is a quasimap version of ~\cite[Conj. ~3.7]{HZ}. 
Let $\cA$ be the cluster algebra associated with $\bfQ$, for which we mainly follow the setting in \cite{FZ-4}. 
We label $\bfI = \{1, \cdots, n\}$, $\bfF = \{n+1, \cdots, m\}$, and $(b_{ij})_{1\leq i\leq m, \, 1\leq j\leq n}$ is the adjacency matrix.
The $x_k$, $1\leq k\leq n$ are the initial cluster variables, and $x'_k = \mu_k (x_k):=\frac{1}{x_k}\left(\prod_{b_{ik}>0}x_i^{b_{ik}}+\prod_{b_{ik}<0}x_i^{-b_{ik}} \right)$ are their \emph{adjacent} mutations. 
Variables $x_i$, $n+1\leq i\leq m$ are the coefficients, or frozen variables, which for us come from the equivariant parameters associated with frozen nodes.
Let $\QH_\sT^\poly (X) [t] \otimes_{\bC[Q]} \bC[ \zeta^{\pm 1} ]$ be the polynomial quasimap quantum cohomology ring of $X$, where  parameters $\zeta_i$'s are related to the K\"ahler parameters via $Q^{(k)} = (-1)^{\bv_k^- - \bv_k } \prod_{i=1}^{m} \zeta_i^{- b_{ik}}$, $1\leq k\leq n$. 

Then there is an injective algebra homomorphism from the cluster algebra $\cA$ to the polynomial quasimap quantum cohomology as follows.



\begin{Theorem}[Thm. ~\ref{Thm-emb}]\label{introd:thmlcustertoqcoh}
If the quiver $\bfQ$ is oriented-acyclic, i.e. does not have oriented cycles, then there is an injective homomorphism of algebras 
$$
\psi: \cA \hookrightarrow \QH_\sT^\poly (X) [t] \otimes_{\bC[Q]} \bC[ \zeta^{\pm 1} ], 
$$
such that 
$$
\psi (x_i) = \zeta_i \cdot c_t (V_i), \qquad 1\leq i\leq m, 
$$
$$
\psi (x'_k) = \zeta_k^{-1} \prod_{i: \, b_{ik} >0} \zeta_i^{ b_{ik}}  \cdot  \delta_t (V_k^- , V_k ) + \zeta_k^{-1} \prod_{i: \, b_{ik} <0} \zeta_i^{-b_{ik}} \cdot \delta_t (V_k^+ , V_k) , \qquad 1\leq k\leq n.
$$
\end{Theorem}

Applying our main theorem, several additional properties are also available. 
The $\bg$-vectors \cite[(6.4)]{FZ-4} of the cluster algebra are related to the exponents of the $\zeta$-parameters in the $Q$-expansions of the cluster variables (see Prop. ~\ref{Prop-g-vector}).
For quivers of type $A$ with nice dimension vectors $\bv$, we can find the images of all \emph{non-initial cluster variables}
(see Prop. ~\ref{Prop-cl-var-A}). 
Finally, we compare our results with the case of actual quantum cohomology $\QH_\sT^\GW (X)$ in the case of \emph{quiver flag varieties} \cite{Cra}. 
In view of the work of \cite{GK}, for quiver flag varieties with nice dimension vectors, our results also lead to a cluster algebra structure in $\QH_\sT^\GW (X) [t] \otimes_{\bC[Q]} \bC[\zeta^{\pm 1}]$ (Cor. ~\ref{Cor-isom-QH-GW}). 
In particular, in the case of partial flag varieties, this recovers the \cite[Thm. ~1.3]{HZ}.

\begin{Remark}
Cluster structures are also found in quantum $K$-theories, in various context, such as Nakajima quiver varieties. 
In \cite{PSZ, KPSZ, KZ-3d}, the quantum $K$-theory rings are described in terms of Bethe ansatz and related to integral systems via the $QQ$-system and $q$-oper formalism. 
In this way, cluster structures are observed in \cite{KZ-q}. 
It will be interesting to study different realizations of cluster algebras in quantum cohomology/$K$-theory. 
We will explore that in future works.
\end{Remark}

\subsection{Acknowledgements}

The authors would like to thank Sergey Fomin, Weiqiang He and Yaoxiong Wen for helpful discussions and collaborations.
We also thank Ionut Ciocan-Fontanine for very helpful communications which help us understand equivalent definitions of the quasimap quantum cohomology. 
Special thanks go to Jiarui Fei for organizing the workshop on Cluster Algebra and Enumerative Geometry, where we learned a lot about cluster algebras.
The work of Y.~Z. is supported by the NSFC grant 12301080, and the startup grant from Shanghai Jiao Tong University. 
Z.~Z.~is supported by NSFC grant 12401077, and the startup grant from Shanghai Jiao Tong University.

\vspace{4ex}

\section{Quasimap quantum cohomology}\label{sec:quasimapqcoh}

In this section, we recall the definition of $(0+)$-stable quasimaps, and the invariants defined by them. 
Our definitions are exactly those given in \cite{CKM, CK-wall, CK-higher-g, CK-big-I, KL}, as well as some natural variants \cite{Cio-notes}.

\subsection{Quasimaps and $I$-functions} \label{Sec-quasimap}

Let $\sG$ be a reductive group over $\bC$, with $\pi_1(\sG)$ torsion free, and $\sN$ be a faithful $\sG$-representation. 
Let $\sT$ be a torus acting linearly on $\sN$, which commutes with $\sG$.
Let $\theta \in \chr(\sG)$ be a chosen stability condition. 
We make the following assumptions:
\begin{itemize}
\setlength{\parskip}{1ex}


\item $\theta$ is \emph{generic} in the sense that
$
\sN^{\sG \text{-s}} = \sN^{\sG \text{-ss}}$ and they are nonempty.

\item The $\sG$-action on $\sN^{\sG \text{-s}}$ is free. 

\end{itemize}
The GIT quotient
$$
X := \sN /\!\!/_\theta \sG = \sN^{\sG \text{-s}} / \sG
$$
is a smooth quasi-projective variety. 
Let $\fX := [\sN / \sG]$ be the stacky quotient, and $X_0 = \sN /\!\!/_0 \sG$ be the affine quotient. 
Then $X \subset \fX$ can be viewed as an open subscheme, and the affinization map $X \to X_0$ is projective. 
The $\sT$-action naturally descends to $X$.
The stability condition $\theta$ naturally defines an ample line bundle $L_\theta$ on $X$.

\begin{Definition}[{\cite[Definition 2.10]{CK-big-I}}]
A \emph{$(0+)$-stable quasimap} to $X$ of genus $g$ and degree\footnote{In existing literature such as \cite{CKM}, $\deg f$ is often defined as an element in $\Hom (\Pic_\sG (\sN), \bZ )$. When $\pi_1 (\sG)$ is torsion-free, this is equivalent to our definition.} $\beta \in \pi_1 (\sG)$, with $n$ \emph{heavy markings} and $m$ \emph{light markings} is a tuple
$$
(C, x_1, \cdots, x_n, y_1, \cdots, y_m, f: C\to \fX)
$$
where
\begin{itemize}
\setlength{\parskip}{1ex}
\item $C$ is a connected projective nodal curve of genus $g$. 
$f$ is of degree $\beta$.

\item $x_i$, $1\leq i \leq n$ (resp. $y_j$, $1\leq j \leq m$) are nonsingular points of $C$, referred to as the heavy markings (resp. light markings). 
One requires that heavy markings are dintinct and do not meet with light markings, while light markings can coincide among themselves.

\item $f^{-1} (\fX \backslash X) \subset C$ is $0$-dimensional and disjoint with all nodes and heavy markings.

\item For any $\epsilon_1, \epsilon_2 \in \bQ_{>0}$, the $\bQ$-line bundle
$$
\omega_C \Big( \sum_{i=1}^n x_i + \epsilon_2 \sum_{j=1}^m y_j \Big)  \otimes f^* (L_\theta)^{\epsilon_1}
$$
is ample. 
\end{itemize}
\end{Definition}

Since we will only consider the $(0+)$-stability, for simplicity we will only refer to them as quasimaps.
The moduli space of such quasimaps is denoted by $Q_{g, n|m} (X, \beta)$. 
It is shown in \cite{CKM} that it is a Deligne--Mumford stack, proper over $X_0$, which carries a natural perfect obstruction theory. 
The standard construction \cite{BF} applies and there is a $\sT$-equivariant virtual fundamental class
$$
[Q_{g, n|m} (X, \beta)]_\vir \in H_*^{\BM, \sT} (Q_{g, n|m} (X, \beta)). 
$$
There are two types of evaluation maps
$$
\ev_i: Q_{g, n|m} (X, \beta) \to X, \qquad 1\leq i\leq n; \qquad \widetilde\ev_j : Q_{g, n|m} (X, \beta) \to \fX, \qquad 1\leq j \leq m. 
$$
Since $X\to X_0$ is projective, the heavy evaluation maps $\ev_i$ are proper. 



Similarly, a \emph{graph-space quasimap} is a quasimap $f: C\to \fX$, together with distinguished irreducible component $C_0 \subset C$, and a chosen isomorphism $C_0 \cong \bP^1$. 
Other conditions of the definition are similar, while the last ampleness condition is only required for the line bundle restricted to $\overline{C \backslash C_0}$. 
The moduli space of graph-space quasimaps are denoted by $QG_{g, n|m} (X, \beta)$, with similar evaluation maps. 

Let $\bC^*_\bh$ be the $1$-dimensional torus scaling $C_0 \cong \bP^1$, where $\bh \in H^*_{\bC^*_\bh} (\pt)$ is identified with the tangent weight $\bh = c_1 (T_0 \bP^1)$. 
The action naturally extends to $QG_{0, 0|0} (X, \beta)$. 
There is a distinguished connected component
$$
F_\beta \subset QG_{0, 0|0} (X, \beta)^{\bC^*_\bh}, 
$$
consisting of quasimaps where the entire degree $\beta$ is carried by a base point at $0\in \bP^1$. 
Since $\infty \in \bP^1$ is not a base point, we have evaluation maps
$$
\ev_\infty: F_\beta \to X; \qquad \widetilde\ev_0 : F_\beta \to \fX. 
$$
where $\ev_\infty$ is proper. 
The $\bC^*_\bh$-equivariant perfect obstruction theory yields a virtual class $[F_\beta]^\vir$, and a virtual normal bundle $N_{F_\beta}^\vir$.

Let $\sK\subset \sG$ be a maximal torus, and denote by $\sX^*(\sK)$ the weight lattice. 
Recall that
$$
H^*_\sT (\fX) \cong H^*_{\sG \times \sT} (\pt) = H^*_\sT (\pt) \otimes_\bC \bC [\sX^* (\sK )]^\sW.
$$
As $\bC [\sX^* (\sK)]$ can be identified with functions $\bC[\Lie \sK]$, we denote by
$$
\xi = (\xi_1, \cdots, \xi_k), \qquad k = \dim_\bC \sK
$$
the collection of a basis of $\sX^*(\sK)$. 
Then $\bC [\sX^* (\sK)] \cong \bC[\xi_1, \cdots, \xi_k]$, and its elements are polynomials
$$
\tau(\xi) = \tau (\xi_1, \cdots, \xi_k).
$$
A general element in $H^*_\sT (\fX)$ is then a polynomial $\tau(\xi)$, symmetric under $\sW$, with coefficients in $H_\sT^* (\pt)$.

\begin{Definition}[{\cite[\S 5.1]{CK-big-I}}] \label{Defn-I}
Let $t , \tau(\xi) \in H^*_\sT (\fX)$.
\begin{itemize}
\setlength{\parskip}{1ex}

\item The \emph{big $\bI$-function} is defined as
$$
\bI (Q; \bh, t) := \sum_{\beta \in \Eff(X)} \ev_{\infty*} \left( \exp \Big( \frac{\widetilde\ev_0^* t}{\bh} \Big) \cdot \frac{[F_\beta]_\vir}{e_{\bC^*_\bh} (N_{F_\beta}^\vir)} \right) \cdot Q^\beta  \ \in \ H^*_{\sT \times \bC^*_\bh} (X)_\loc[\![Q ]\!],
$$
where $(-) [\![Q ]\!]$ denotes the formal power series ring over the cone of effective curve classes $\Eff(X)$ in $X$, and $(-)_\loc$ means to pass to the localized theory $(-)_\loc := (-) \otimes_{H_\sT^* (\pt)} \Frac (H_{\sT \times \bC^*_\bh}^* (\pt))$. 
The directional derivative of $\bI$ along $\tau(\xi)$ is defined as
$$
\bh \partial_\tau \bI (Q; \bh, t) := \bh \frac{\partial}{\partial s} \bI (Q; \bh, t + s \tau) \Big |_{s=0}. 
$$

\item The \emph{small $I$-function with insertion} $\tau (\xi)$ is defined as
$$
I^{(\tau (\xi) )} (Q; \bh) := \bh \partial_\tau \bI (Q; \bh, t) \Big|_{t =0} = \sum_{\beta \in \Eff(X)} \ev_{\infty*} \left( \widetilde\ev_0^* \tau(\xi) \cdot \frac{[F_\beta]_\vir}{e_{\bC^*_\bh} (N_{F_\beta}^\vir)} \right) \cdot Q^\beta.
$$

\item The \emph{small $I$-function} is defined as
$$
I (Q; \bh) := \bI (Q; \bh, 0) = I^{(1 )} (Q; \bh). 
$$

\end{itemize}

\end{Definition}

\begin{Remark} \label{Rk-power-series}
For simplicity, we will refer to small $I$-functions simply as ``$I$-functions".
A more precise notation for the power series ring in $Q^\beta$, $\beta \in \Eff(X)$, with coefficients in a ring $R$, is $R[\![Q^{\Eff(X)}]\!]$. 
For simplicity, we often abbreviate it as $R[\![Q]\!]$, when no confusion arises.
\end{Remark}

Recall that the open immersion $X \subset \fX$ gives the $\sT$-equivariant Kirwan surjection
$$
(-)|_X : H^*_\sT (\fX) \cong H^*_{ \sG \times \sT}(\pt) = H^*_\sT (\pt) \otimes_\bC \bC [\sX^* (\sK)]^\sW \twoheadrightarrow H^*_\sT (X). 
$$
where the image of $\tau(\xi) = \tau(\xi^\lambda, \ \lambda\in \sX^*(\sK))$ is nothing but the tautological class $\tau (c_1 (L_\chi), \chi \ \in \sX^* (\sK))$. 
In particular, the $Q$-constant term of the $I$-function is
$$
I^{(\tau (\xi) )} (Q; \bh)  |_{Q=0} = \tau(\xi) |_X. 
$$

Extended by linearity, $I$-functions with insertions define a map 
$$
I^{(-)} (Q; \bh) : H^*_{\sG \times \sT \times \bC^*_\bh} (\pt) [\![Q]\!] \to H^*_{\sT \times \bC^*_\bh} (X)_\loc[\![Q]\!]
$$
of $\bC [\![Q]\!]$-modules.

\subsection{Quantum tautological classes and quasimap quantum cohomology}

The big $\bI$-function admits a Birkhoff factorization, described as follows. 
Let $\gamma \in H^*_\sT (X)$, $t\in H_\sT^* (\fX)$, and $\psi_i$ be the tautological psi class at the heavy marking $x_i$.
The $\bS$-operators are defined as \cite[\S 4.1]{CK-big-I}
$$
\bS(Q; \bh, t) (\gamma) := \gamma + \sum_{(\beta, m) \neq (0,0)} \frac{Q^\beta}{m!} \ev_{1*} \Big( \frac{\ev_2^* \gamma}{\bh-\psi_1} \prod_{j=1}^m \widetilde\ev_j^* t \cdot [Q_{0, 2|m} (X, \beta) ]_\vir \Big), 
$$
$$
\bS^* (Q; -\bh, t) (\gamma) := \gamma + \sum_{(\beta, m) \neq (0,0)} \frac{Q^\beta}{m!} \ev_{1*} \Big( \frac{\ev_2^* \gamma}{-\bh-\psi_2} \prod_{j=1}^m \widetilde\ev_j^* t \cdot [Q_{0, 2|m} (X, \beta) ]_\vir \Big).
$$
Recall that the evaluation maps of the graph-space moduli space $QG_{0,1|m} (X, \beta)$ 
$$
\ev^{QG}_1: QG_{0,1|m} (X, \beta) \to X \times \bP^1; \qquad \widetilde\ev^{QG}_j : QG_{0,1|m} (X, \beta) \to \fX \times \bP^1, \qquad 1\leq j\leq m. 
$$
with $\ev_1 := \pr_1 \circ \ev^{QG}_1$ and $\widetilde\ev_j := \pr_1 \circ \widetilde\ev^{QG}_j$. 
Denote by $p_0, p_\infty \in H^*_{\bC^*_\bh} (\bP^1)$ the classes given by the fixed points $0$ and $\infty$. 
The $P$-series is defined as
$$
P (Q; \bh, t) := \sum_{\beta, m} \frac{Q^\beta}{m!} \ev_{1*} \Big( (\ev_1^{QG})^* (1 \otimes p_\infty) \prod_{j=1}^m \widetilde\ev_j^* t \cdot [QG_{0,1|m} (X, \beta)]_\vir \Big).
$$
Given $\tau(\xi) \in H_\sT^* (\fX)$, its variant is \cite[\S 2.3]{KL}
$$
P_\tau (Q; \bh, t) := \sum_{\beta, m} \frac{Q^\beta}{m!} \ev_{1*} \Big( (\ev_1^{QG})^* (1 \otimes p_\infty) ( \widetilde\ev_1^{QG} )^* (\tau \otimes p_0) \prod_{j=1}^m \widetilde\ev_j^* t \cdot [QG_{0,1|m} (X, \beta)]_\vir \Big).
$$
Note that $\bS (Q; \bh, t) (\gamma) = \gamma + O(\bh^{-1})$, and by properness, $P_\tau (Q;  \bh, t) = \tau |_X + O(\bh)$. 

\begin{Proposition}[{\cite[Proposition 4.2]{CK-big-I}, \cite[\S 2.3]{KL}}] \label{Prop-Birkhoff}
We have
\begin{enumerate}[label=\arabic*)]
\setlength{\parskip}{1ex}

\item $\bS^* (Q; -\bh, t) \circ \bS (Q; \bh, t) = \Id$. 

\item $\bS^* (Q; -\bh, t) (\bh \partial_\tau \bI (Q; \bh, t)) = P_\tau (Q; \bh, t)$. 

\end{enumerate}
\end{Proposition}

\begin{Definition}
Let $\tau (\xi) \in H^*_\sT (\fX) \cong H^*_{ \sG \times \sT} (\pt) = H^*_\sT (\pt) \otimes_\bC \bC [\sX^* (\sK)]^\sW$.
\begin{itemize}
\setlength{\parskip}{1ex}

\item The \emph{capped $I$-function with insertion} $\tau(\xi)$ is defined as
$$
\widehat I^{(\tau(\xi))} (Q; \bh) := P_\tau (Q; \bh, 0) \ \in \ H^*_{\sT \times \bC^*_\bh} (X) [\![Q]\!]. 
$$

\item The \emph{quantum tautological class} associated with $\tau(\xi)$ is the specialization
$$
\widehat{\tau (\xi)} (Q) := \widehat I^{(\tau(\xi))} (Q; \bh) |_{\bh = 0} = P_\tau (Q; 0, 0) \ \in \ H^*_\sT (X) [\![Q]\!].
$$

\end{itemize}

\end{Definition}

As in the usual Gromov--Witten theory, the quasimap invariants with heavy markings defines a CohFT, with state space $H_\sT^*(X)$. 
In particular, the $3$-point functions defines the quantum cohomology.

\begin{Definition}\label{def:quasimapqcoh}
Let $\gamma, \delta, v \in H_\sT^*(X)$. 
\begin{itemize}
\setlength{\parskip}{1ex}

\item The \emph{quasimap big quantum product} is defined as
$$
\gamma *_v \delta := \sum_{\beta\in \Eff(X)} \ev_{1*} \Big( \ev_2^* \gamma \cdot \ev_3^* \delta \cdot [Q_{0, 3|0} (X, \beta)]_\vir \Big) \cdot Q^\beta. 
$$

\item The \emph{quasimap small quantum product} is defined as
$$
\gamma * \delta := \gamma *_0 \delta. 
$$

\item The \emph{quasimap small quantum cohomology} is 
$$
\QH_\sT (X) := (H^*_\sT [\![Q]\!], *),
$$
where the product is extended linearly to $H^*_\sT [\![Q]\!]$.

\end{itemize}
\end{Definition}

It is a commutative associative algebra, which is a deformation of $H^*_\sT (X)$.
For simplicity, we will refer to the quasimap small quantum cohomology simply as ``quasimap quantum cohomology". 
Unlike the usual quantum cohomology, the identity for $\QH_\sT (X)$ here is not the fundamental class $1$, but the \emph{quantum fundamental class} $\widehat{1} (Q)$.

\begin{Remark}  \label{Rk-alt-defn-I}
In an earlier version of this paper \cite{ZZ-v1}, and also in \cite{Oko-lec, PSZ, KPSZ, SZ, KZ-3d} for the $K$-theoretic setting, quasimaps, $I$-functions, quantum tautological classes and quasimap quantum cohomology are defined in an alternative way, which is a specialization of the definition we adopt here. 
We list here a brief comparison between the two approaches, which we learned from Ionut Ciocan-Fontanine \cite{Cio-notes}.
\begin{itemize}
\setlength{\parskip}{1ex}

\item In \cite{ZZ-v1}, one considers the following moduli space of quasimaps
$$
\QM(X, \beta)_{\ns \, \infty} = \{f: \bP^1 \to \fX \mid f^{-1} (\fX \backslash X) \text{ is } 0 \text{-dimensional, and disjoint from } \infty , \ \deg f = \beta \}. 
$$
It carries a natural perfect obstruction theory, and evaluation maps $\ev_\infty$ mapping to $X$ (resp. $\widetilde\ev_0$ mapping to $\fX$). 
Its fixed locus is 
$$
\QM(X, \beta)_{\ns \, \infty}^{\bC^*_\bh} = F_\beta,
$$
and the restriction of $\ev_\infty$, $\widetilde\ev_0$ to the fixed locus coincides with those of $F_\beta$. 
Hence the $I$-function with insertion coincides with the same notion here, i.e.
$$
I^{(\tau (\xi) )} (Q; \bh) =  \sum_{\beta \in \Eff(X)} \ev_{\infty*} \left( \widetilde\ev_0^* \tau(\xi) \cdot [\QM(X, \beta)_{\ns \, \infty}]_\vir \right) \cdot Q^\beta,
$$
where the pushforward $\ev_{\infty*}$ is defined in terms of $\bC^*_\bh$-localization.

\item $\QM(X, \beta)_{\rel \, p_1, \cdots, p_n}$ in \cite{ZZ-v1} coincides with $QG_{0, n|0} (X, \beta) \cap \prod_{i=1}^n (\ev^{QG}_i)^{-1} (p_i)$ here. 
Moreover, when $n=3$, the map
$$
\QM(X, \beta)_{\rel \, p_1, p_2, p_3} \to Q_{0, 3|0} (X, \beta)
$$
which forgets the parametrization of the distinguished component $C_0 \cong \bP^1$ is an isomorphism, preserving the virtual classes.
The degeneration/gluing formula there follows from the WDVV equation of the quasimap CohFT. 
In particular, the quasimap quantum cohomology $\QH_\sT (X)$, capped $I$-function with insertion and quantum tautological class $\widehat\tau(Q)$ coincide with the same notions here.

\item The capping operator $\Psi(Q; \bh)$ in \cite{ZZ-v1} coincides with a specialization of the $\bS$-operator
$$
\Psi(Q; \bh) = \bS^*(Q; -\bh, 0). 
$$ 
In particular, the identity 
$$
\widehat I^{(\tau(\xi))} (Q; \bh) = \Psi (Q; \bh) I^{(\tau (\xi))} (Q; \bh)
$$ 
is the specialization at $t=0$ of the Birkhoff factorization, i.e. Proposition \ref{Prop-Birkhoff} 2). 

\end{itemize}
\end{Remark}

\begin{Lemma} \label{Lemma-product}
Given $\tau (\xi), \sigma(\xi) \in H^*_{ \sG \times \sT} (\pt)$, we have
$$
\widehat{\tau(\xi) \sigma(\xi)} (Q) = \widehat{\tau(\xi)}(Q) * \widehat{\sigma (\xi)} (Q). 
$$
\end{Lemma}

\begin{proof}
This is \cite[Lemma ~4.16]{SZ}, whose proof works in general.
\end{proof}

\begin{Proposition} \label{Prop-QH-I}
Let $H^*_{ \sG \times \sT} (\pt) [\![Q]\!]$ be equipped with the usual ring structure.
\begin{enumerate}[label=\arabic*)]
\setlength{\parskip}{1ex}

\item The $H^*_\sT (\pt) [\![Q]\!]$-linear map 
\begin{equation} \label{eqn-quantum-taut}
H^*_{\sG \times \sT} (\pt) [\![Q]\!] \to H^*_\sT (X) [\![Q]\!], \qquad H^*_{ \sG \times \sT} (\pt) \ni \ \tau(\xi) \mapsto \widehat{\tau(\xi)} (Q)
\end{equation}
is a surjective map of $H^*_\sT (\pt) [\![Q]\!]$-algebras.

\item If $F(\xi, Q; \bh) \in H^*_{\sG \times \sT \times \bC^*_\bh} (\pt) [\![Q]\!]$ satisfies $I^{( F(\xi, Q; \bh) )} (Q; \bh) = 0$, then
$$
\widehat {F(\xi, Q; 0)} (Q) = 0.
$$
In other words, $F(\xi, Q; 0)$ lies in the kernel of \eqref{eqn-quantum-taut}.

\end{enumerate}

\end{Proposition}

\begin{proof}
1): The surjectivity follows from Kirwan surjectivity and Nakayama's lemma.
Lemma ~\ref{Lemma-product} then shows that it is a ring map. 

2): Suppose that $I^{(F(\xi, Q; \bh))} (Q; \bh) = 0$. 
Then Proposition \ref{Prop-Birkhoff} 2) implies that $\widehat I^{( F(\xi, Q; \bh))} (Q; \bh) = 0 \in H^*_{\sT \times \bC^*_\bh} (X) [\![Q]\!]$. 
Write $F(\xi, Q; \bh) = \sum_{n=0}^N F_n (\xi, Q) \bh^n$. 
We have
$$
0 = \widehat I^{(F(\xi, Q; \bh))} (Q; \bh) = \widehat I^{(F (\xi, Q; 0))} (Q; \bh)  + \sum_{n=1}^N \widehat I^{(F_n (\xi, Q))} (Q; \bh) \bh^n. 
$$
Each term $\widehat I^{(F_n (\xi, Q))} (Q; \bh) \bh^n$ lies in $\bh^n \cdot H^*_{\sT \times \bC^*_\bh} (X) [\![Q]\!]$. 
Set $\bh=0$, and we obtain $\widehat {F(\xi, Q; 0)} (Q) = \widehat I^{(F (\xi, Q; 0))} (Q; \bh) |_{\bh = 0} = 0$.
\end{proof}

\subsection{Abelianization of $I$-functions}

Recall that $\sK \subset \sG$ is the maximal torus. 
One can often study the geometry of $X$ in terms of the \emph{abelianization}, i.e.
$$
X^\ab := \sN /\!\!/_\theta \sK . 
$$
We assume that $\theta$ is chosen such that $\sN^{\sK \text{-ss}} = \sN^{\sK \text{-s}}$ and the $\sK$-action on $\sN^{\sK \text{-s}}$ is free. 
Then $X^\ab$ is a smooth quasi-projective toric variety.
By definition, it fits into the diagram
\[
\xymatrix{
Y := & \sN^{\sG \text{-s}} / \sK \ar[d]_\pi \ar@{^(->}[r]^-j & \sN^{\sK \text{-s}} / \sK  & =  X^\ab \\
X  = & \sN^{\sG \text{-s}} / \sG
}
\]
where $\pi$ is a $\sG / \sK$-fibration, and $j$ is an open immersion. 
Everything here is compatible with the $\sT$-action.

$H^*_\sT (Y)$ and $H^*_\sT (X^\ab)$ admit natural $\sW$-actions.
Fixing a Borel subgroup $\sB\subset \sG$, the map 
$\pi$ factors as an affine fibration followed by a $\sG / \sB$-fibration. 
The pullback $\pi^*$ identifies
$$
\pi^* : H^*_\sT (X) \cong H^*_\sT (Y )^\sW. 
$$

Let $\Phi$ be the set of all roots of $\sG$, and $\Phi_+$ be the positive roots. 
For each $\alpha \in \Phi$, there is an associated line bundle $L_\alpha$ on $X^\ab$.
$$
e = \prod_{\alpha\in \Phi_+} \alpha, \qquad  e|_{X^\ab} = \prod_{\alpha \in \Phi_+} c_1 (L_\alpha). 
$$
Cohomologies of the non-abelian and abelian quotients are related as stated in the following lemma.
The proof in the non-equivariant case is due to \cite{Bri, ES, Mar}, whose extension to the equivariant setting is straightforward.
See also \cite[~\S 3.1]{CKS}.

\begin{Lemma} \label{Lemma-ab-H}
\begin{enumerate}[label=\arabic*)]
\setlength{\parskip}{1ex}

\item $\ann (e|_{X^\ab}) \subset \ker j^*$. 

\item The map $(\pi^*)^{-1} \circ j^*$ induces an isomorphism of rings
$$
(\pi^*)^{-1} \circ j^*: \quad  ( H^*_\sT (X^\ab) / \ann (e|_{X^\ab}) )^\sW \cong  H^*_\sT (X). 
$$

\item If $H^*_\sT (X^\ab) \cong H^*_{\sK \times \sT} (\pt) / J^\ab$ for some ideal $J^\ab$, then there is a presentation
$$
H^*_\sT (X) \cong H^*_{\sK\times \sT} (\pt)^\sW / J, \qquad \text{where }
J = \Big\la e^{-1} \sum_{w\in \sW} (-1)^{l(w)} w \cdot f \ \Big| \ f\in J^\ab \Big\ra. 
$$
\end{enumerate}
\end{Lemma}

The $I$-functions are also related.
Let $\sX_*(\sK)$ be the cocharacter lattice. 
Recall that $\pi_1 (\sG) \cong \sX_* (\sK) / \bZ \Phi$, where $\bZ \Phi$ is the root lattice for $\sG$.
For any $d\in \sX_* (\sK)$, we denote its image by
$$
d \ \mapsto \ \bar d \ \in \ \pi_1 (\sG) \cong  \sX_* (\sK) / \bZ \Phi.
$$
By \cite[Lemma ~5.1.1]{Web}, 
$\Eff(X^\ab)$ maps onto $\Eff(X)$. 
Denote
$$
\gamma (\xi; \bh) := \prod_{\alpha \in \Phi} \prod_{l =0}^\infty (\alpha - l \bh). 
$$

\begin{Proposition} \label{Prop-abelianization-I}
Let $\tau(\xi) \in H^*_\sT(\pt) \otimes_\bC H^*_\sK(\pt)^\sW$. 
We have
$$
 I^{(\tau(\xi))}_X (Q; \bh) = (\pi^*)^{-1} j^* \left( \frac{1}{\gamma (\xi; \bh) |_{X^\ab} } \cdot  I^{( \gamma (\xi; \bh) \cdot \tau(\xi))}_{X^\ab} (\widetilde Q; \bh) \right) \Big|_{\widetilde Q \mapsto Q}, 
$$
where $\widetilde Q \mapsto Q$ means the specialization of K\"ahler parameters
$$
\widetilde Q^{d} \mapsto Q^{\bar d}.
$$
\end{Proposition}

\begin{proof}
\cite[Thm. ~1.1.1]{Web} can be generalized in a straightforward way to $I$-functions with insertions, and we have 
$$
I_{X, \bar d}^{(\tau(\xi))} (\bh) = (\pi^*)^{-1} j^* \sum_{d \mapsto \bar d} \left( \prod_{\alpha \in \Phi}  \frac{ \prod_{l=0}^\infty (\alpha + \la \alpha, d\ra \bh  - l\bh ) 
}{ 
\prod_{l=0}^\infty (\alpha  - l \bh ) 
} \right) \Bigg|_{X^\ab} \cdot I_{X^\ab, d}^{(\tau(\xi))} (\bh) , 
$$
where $I_{X, \bar d}^{(\tau(\xi))} (\bh)$ is the $Q^{\bar d}$-coefficient of the $I$-function of $X$, and similarly for $I_{X^\ab, d}^{(\tau(\xi))} (\bh)$.

Note that for any $\bC^*_\bh$-equivariant line bundle $L$ on $\bP^1$, we have $L|_0 - L_\infty = (\deg L)\bh$. 
So 
$$
\widetilde\ev_0^* \gamma (\xi; \bh) = \ev_\infty^* \gamma (\xi; \bh) |_{\alpha \mapsto \alpha + \la \alpha, d \ra \bh}.
$$
The result then follows from a summation over degrees.
\end{proof}

\begin{Remark}
The function $\gamma (\xi; \bh)$ is an infinite product in terms of $\bh$. 
However, for each degree $d$, all but finitely many factors from $\gamma (\xi; \bh)$ in the insertion cancel with the prefactor $\dfrac{1}{\gamma (\xi; \bh) |_{X^\ab} }$. 
So the resulting function only involves a finite product.
\end{Remark}

For later use, we also need to consider quasimaps to the intermediate space $Y = \sN^{\sG\text{-s}} / \sK$, which is also an open subscheme of $\fX = [\sN / \sK]$. 
We have a $\sT$-equivariant Kirwan surjection
$$
(-)_Y = j^* \circ (-)|_{X^\ab} : \quad H_{\sK \times \sT}^*(\pt) \twoheadrightarrow H^*_\sT (Y). 
$$
$I$-functions with insertions can also be defined for $Y$, in the same way as above.

\begin{Lemma}
\begin{enumerate}[label=\arabic*)]
\setlength{\parskip}{1ex}

\item $e |_Y \in H^*_\sT (Y)$ is not a zero divisor. 

\item Let $\tau(\xi) \in H^*_{\sK \times \sT} (\pt)$. 
Then
$$
j^* I^{(\tau(\xi))}_{X^\ab} (\widetilde Q; \bh) =  I^{(\tau(\xi))}_{Y} (\widetilde Q; \bh) . 
$$
\end{enumerate}
\end{Lemma}

\begin{proof}
1) follows from the fact that $\ann (e|_{X^\ab}) \subset \ker j^*$. 

For 2), for convenience, we use the alternative definition of $I$-functions given in Remark \ref{Rk-alt-defn-I}. 
Note that under the open immersion $j_{\QM}: \QM(Y, d)_{\ns \, \infty} \hookrightarrow \QM(X^\ab, d)_{\ns \, \infty}$, we have 
$$
j_{\QM}^* [\QM(X^\ab, \beta)_{\ns \, \infty}]_\vir = [\QM(Y, \beta)_{\ns \, \infty}]_\vir.
$$  
Moreover, recall part of \cite[Diag. ~(31)]{Web} that after taking $\bC^*_\bh$-fixed loci, there is a Cartesian diagram
\[
\xymatrix{
\QM(Y, d)_{\ns \, \infty}^{\bC^*_\bh} \ar@{^(->}[d]_-{\ev^Y_\infty} \ar@{^(->}[r] & \QM(X^\ab, d)_{\ns \, \infty}^{\bC^*_\bh} \ar@{^(->}[d]^-{\ev_\infty} \\
Y \ar@{^(->}[r]^-j & X^\ab, 
}
\]
where $d\in \Eff(X^\ab)$. 
Therefore, the degree $d$ contribution is
\ben
j^* I^{(\tau(\xi))}_{X^\ab, d} (\bh) &=& j^* \ev_{\infty *} \left( \ev_0^* \tau (\xi) \cdot [\QM(X^\ab, d)_{\ns \, \infty}]_\vir \right) \\
&=& \ev_{\infty *}^Y j^* \left( \ev_0^* \tau (\xi) \cdot [\QM(X^\ab, d)_{\ns \, \infty}]_\vir \right) \\
&=& \ev_{\infty *}^Y \left( \ev_0^{Y*} \tau (\xi) \cdot [\QM(Y, d)_{\ns \, \infty}]_\vir \right) \\
&=& I^{(\tau(\xi))}_{Y, d} (\bh),
\een
where we have used the fact $\ev_0^Y = \ev_0 \circ j$.  
\end{proof}

\subsection{$\QH_\sT (X^\ab)$ via quantum differential equation}

It is a general feature that the $I$-function satisfies a quantum differential equation, which can be explicitly written down in the abelian case. 
In this section, we will consider insertions of the form $\tau (\xi) \in H^*_{\sK \times \sT} (\pt)$ (not necessarily $\sW$-symmetric), and call the $I$-function
$$
 I^{(\tau(\xi))}_{X^\ab} (\widetilde Q; \bh)
$$
an \emph{abelian $I$-function} with insertion.

\begin{Definition}
Let $\tau(\xi) \in H^*_{\sK \times \sT} (\pt)$. 
Define a renormalized version of the abelian $I$-function as
$$
\widetilde I^{(\tau(\xi))}_{X^\ab} (\widetilde Q; \bh) := \widetilde Q^{ \xi |_{X^{\ab}} / \bh} \cdot I^{(\tau(\xi))}_{X^\ab} (\widetilde Q; \bh) ,
$$
where the prefactor $\widetilde Q^{ \xi / \bh}$ is a multivalued function of $\widetilde Q$, which means the following.

Let $\{\xi_1, \cdots, \xi_k\}$ be a $\bZ$-basis of $\sX^*(\sK)$, $\{d_1, \cdots, d_k \}$ be its dual basis of $\sX_*(\sK)$.
Then
$$
\widetilde Q^{ \xi / \bh} := \prod_{j=1}^k ( \widetilde Q^{d_j} )^{  \xi_j / \bh}. 
$$
Independent of the choice of basis, it is a formal solution satisfying the differential equation\footnote{For any $\lambda \in \sX^*(\sK)$ and $d\in \sX_*(\sK)$, the differential operator $\lambda \widetilde Q \frac{\partial}{\partial \widetilde Q}$ on $\widetilde Q^d$ is defined as 
$$
\Big( \lambda \widetilde Q \frac{\partial}{\partial \widetilde Q}  \Big) \widetilde Q^d := \la \lambda, d \ra \widetilde Q^d.
$$}
$$
\Big( \lambda \bh \cdot \widetilde Q \frac{\partial}{\partial \widetilde Q} \Big) \widetilde Q^{\xi / \bh}  =  \lambda \cdot \widetilde Q^{ \xi / \bh}, \qquad \lambda \in \sX^*(\sK). 
$$
As a result, $\widetilde Q^{ \xi |_{X^{\ab}}/ \bh}$ is the ``restriction to $X^{\ab}$", satisfying
$$
\Big( \lambda \bh \cdot \widetilde Q \frac{\partial}{\partial \widetilde Q} \Big)  \widetilde Q^{\xi |_{X^{\ab}} / \bh}  =   ( \lambda |_{X^{\ab}}) \cdot  \widetilde Q^{\xi |_{X^{\ab}} / \bh} =  c_1^\sT (L_\lambda) \cdot  \widetilde Q^{ \xi |_{X^{\ab}} / \bh} , \qquad \lambda \in \sX^*(\sK). 
$$
\end{Definition}

An explicit formula for the $I$-function with insertion $1$ can be found in \cite[Thm. ~5.4]{CCK} (for a $K$-theoretic version, see ~\cite[~\S 3.2]{RWZ}), which was originally introduced as a definition of $I$-functions, see e.g. ~\cite{Giv}. 
The same applies to $I$-functions with insertions. 
Let
$$
\sN = \bigoplus_{i=1}^n \bC_{u_i + \lambda_i}, \qquad u_i \in \sX^*(\sT), \ \lambda_i \in \sX^*(\sK), 
$$
be the decomposition into $1$-dimensional $\sK \times \sT$-representations.
Then
\begin{equation} \label{eqn-explicit-I}
\widetilde I^{(\tau(\xi))}_{X^\ab} (\widetilde Q; \bh) = \widetilde Q^{ \xi |_{X^{\ab}} / \bh} \sum_{d\in \Eff(X^\ab)} \prod_{i=1}^n \frac{\prod_{l=0}^\infty (u_i + c_1 (L_{\lambda_i}) - l\bh)}{\prod_{l=0}^\infty (u_i + c_1 (L_{\lambda_i})  + \la \lambda_i, d \ra \bh - l \bh)} \cdot \tau(\xi + d\bh) |_{X^{\ab}} \cdot \widetilde Q^d,
\end{equation}
where the factor $\tau (\xi + d \bh)$ means the following.
Let $\xi_1, \cdots, \xi_k$ be a $\bZ$-basis, and write $\tau(\xi) = \tau(\xi_1, \cdots, \xi_k)$. 
Then
$$
\tau(\xi + d\bh) := \tau \left( \xi_1 + \la \xi_1, d \ra \bh, \cdots, \xi_k + \la \xi_k , d \ra \bh \right) , 
$$
and $\tau(\xi + d\bh) |_X$ means its restriction under the $\sT$-equivariant Kirwan surjection.

The following lemma shows how the differential operators act on the $I$-functions, which follows from a straightforward and explicit computation.

\begin{Lemma} \label{Lemma-diff-oper-I}
Given $\lambda \in \sX^* (\sK)$, 
$$
\Big( \lambda \bh \cdot \widetilde Q \frac{\partial}{\partial \widetilde Q} \Big) 
\, \widetilde I^{(\tau(\xi))}_{X^\ab} (\widetilde Q; \bh)  = \widetilde I^{(\lambda \cdot \tau(\xi))}_{X^\ab} (\widetilde Q; \bh) . 
$$
In particular, let $\xi_1, \cdots, \xi_k$ be a $\bZ$-basis and write $\tau(\xi) = \tau(\xi_1, \cdots, \xi_k)$.
Then
$$
\widetilde I^{(\tau(\xi))}_{X^\ab} (\widetilde Q; \bh)  = \tau \Big(  \xi_1 \bh \cdot \widetilde Q \frac{\partial}{\partial \widetilde Q} , \cdots,  \xi_k \bh \cdot \widetilde Q \frac{\partial}{\partial \widetilde Q} \Big) 
\, \widetilde I^{(1)}_{X^\ab} (\widetilde Q; \bh) .
$$
\end{Lemma}

Now we arrive at an explicit presentation for a system of quantum differential equations satisfied by $\widetilde I^{(1)}_{X^\ab} (\widetilde Q; \bh)$, and the quasimap quantum cohomology $\QH (X^\ab)$\footnote{This is exactly Batyrev's presentation \cite{Bat}, but it works for all toric GIT quotients.}.

\begin{Proposition} \label{Prop-diff-QH-ab}
\begin{enumerate}[label=\arabic*)]
\setlength{\parskip}{1ex}

\item The abelian $I$-function $\widetilde I^{(1)}_{X^\ab} (\widetilde Q; \bh)$ is annihilated by the quantum differential operators
$$
\prod_{i: \la \lambda_i, d \ra >0} \prod_{m = 0}^{\la \lambda_i, d \ra -1} \Big(  \lambda_i \bh \cdot \widetilde Q \frac{\partial}{\partial \widetilde Q} + u_i - m\bh \Big) - 
\widetilde Q^d  
\prod_{i: \la \lambda_i, d \ra < 0} \prod_{m = 1}^{- \la \lambda_i, d \ra } \Big(  \lambda_i \bh \cdot \widetilde Q \frac{\partial}{\partial \widetilde Q} + u_i - m\bh \Big), 
$$
for all $d\in \Eff(X^\ab) \subset \sX_*(\sK)$.
\item We have
$$
\QH_\sT (X^\ab) \cong H_{\sK\times \sT}^* (\pt) [\![ \widetilde Q]\!] / \cJ^\ab,
$$ 
where the ideal $\cJ^\ab$ is generated by
\begin{equation} \label{eqn-QH-rel-ab}
\prod_{i: \la \lambda_i, d \ra >0} (u_i + \lambda_i)^{\la \lambda_i, d\ra} - \widetilde Q^d \prod_{i: \la \lambda_i, d \ra < 0} ( u_i + \lambda_i )^{- \la \lambda_i, d\ra}, 
\end{equation}
for all $d\in \Eff(X^\ab) \subset \sX_*(\sK)$.
\end{enumerate}
\end{Proposition}

\begin{proof}
1) follows from a straightforward computation. 
It can be found in \cite{Giv}; see also \cite[\S 4.2]{Iri-int}. 
2) follows from 1), together with Prop. ~\ref{Prop-QH-I} and Lemma ~\ref{Lemma-diff-oper-I}.
\end{proof}

\begin{Remark}
For the quantum differential equations and quantum multiplication relations in Prop. ~\ref{Prop-diff-QH-ab}, it suffices to consider those where $d$ is a $\bZ$-generator of $\Eff(X^\ab)$.
In the language of toric varieties, $\Eff(X^\ab)$ is the Mori cone of $X^\ab$, and the minimal generating class consists of the primitive collections of the defining fan ~\cite{Bat}. 
Hence there are actually only finitely many independent quantum differential equations, and $\cJ^\ab$ is finitely generated.
\end{Remark}

\begin{Remark}
Note that in the above presentation, we are using the polynomials in the domain ring $H_\sK^* (\pt)$, i.e. the $\tau (\xi)$ in the map \eqref{eqn-quantum-taut}: $\tau(\xi) \mapsto \widehat{\tau(\xi)} (\widetilde Q)$, rather than the quantum tautological class $\widehat{\tau(\xi)} (\widetilde Q)$. 
In particular, the identity element is still $1$, not $\widehat{1} (\widetilde Q)$.
\end{Remark}

\begin{Remark} \label{Rk-quantum-D}
An alternative and deeper approach to understanding the above presentations is to consider the \emph{quantum $D$-module}; see e.g. ~\cite{Iri}. 
Let 
$$
\cD := H^*_\sT (\pt) [\bh] \Big\la \lambda \bh \cdot \widetilde Q \frac{\partial}{\partial \widetilde Q}, \ \lambda \in \sX^* (\sK); \ \widetilde Q^d, \ d \in \sX_*(\sK) \Big\ra, 
$$
be the ring of differential operators in $\widetilde Q$. 
The commutation relation
$$
\Big[   \lambda \bh \cdot \widetilde Q \frac{\partial}{\partial \widetilde Q} , \ \widetilde Q^d \Big] = \la \lambda, d \ra \bh \cdot  \widetilde Q^d
$$
implies that there is an increasing filtration of $\cD$, given by the order of $\lambda \bh \cdot \widetilde Q \dfrac{\partial}{\partial \widetilde Q}$'s.
Its associated graded ring is the polynomial ring
$$
\gr \cD = H_\sT^*(\pt) \left[ \lambda , \  \lambda \in \sX^* (\sK); \ \widetilde Q^d, \ d \in \sX_*(\sK) \right], 
$$
where the \emph{symbol map} sends 
$$
\lambda \bh \cdot \widetilde Q \frac{\partial}{\partial \widetilde Q} \mapsto \lambda, \qquad \bh \mapsto 0.
$$
The relations in Prop. ~\ref{Prop-diff-QH-ab} 1) form a left ideal $\cJ^\ab_\nc$ in $\cD$, whose quotient $\cM^\ab := \cD / \cJ^\ab_\nc$ is a left $\cD$-module, which we may refer to as the \emph{quantum $D$-module generated by} $\widetilde I^{(1)}_{X^\ab} (\widetilde Q; \bh)$. 
The relations in Prop. ~\ref{Prop-diff-QH-ab} 2) form the associated graded ideal $\gr \cJ^\ab_\nc \subset \gr \cD$. 
We have $\QH_\sT (X^\ab) \cong \gr \cD / \gr \cJ^\ab_\nc = \gr \cM^\ab$. 
\end{Remark}

\begin{Lemma} \label{Lemma-0-div}
Let $F(\xi, \widetilde Q; \bh), G(\xi,  \widetilde Q; \bh) \in H_{\sK \times \sT \times \bC^*_\bh}^*(\pt) [\![\widetilde Q]\!]$. 
Suppose that 
\begin{enumerate}[label=(\arabic*)]
\setlength{\parskip}{1ex}
\item $\widehat{F(\xi,  \widetilde Q; 0)} ( \widetilde Q)$ is not a zero divisor in $\QH_\sT (X^\ab)$; 

\item $I_{X^\ab}^{(F(\xi,  \widetilde Q; \bh) G(\xi,  \widetilde Q; \bh))} ( \widetilde Q; \bh) = 0$. 
\end{enumerate}
Then $I_{X^\ab}^{( G(\xi,  \widetilde Q; \bh))} ( \widetilde Q; \bh) = 0$.

Similarly, suppose that $\tau (\xi) \in H^*_{\sK \times \sT } (\pt)$ satisfies $\tau (\xi)|_Y$ is not a zero divisor in $H_\sT^* (Y)$, and $I_Y^{(\tau(\xi) G(\xi, \widetilde Q; \bh))} (\widetilde Q; \bh) = 0$. 
Then $I_Y^{(G(\xi, \widetilde Q; \bh))} (\widetilde Q; \bh) = 0$. 
\end{Lemma}

\begin{proof}
We use the language of quantum $\cD$-modules in Rem. ~\ref{Rk-quantum-D}. 
Let $1 \in \cM^\ab$ be the image of $1\in \cD$ in the quantum $\cD$-module. 
Condition (ii) implies that 
\begin{equation} \label{eqn-FG-0}
F \Big( \lambda \bh \cdot \widetilde Q \frac{\partial}{\partial \widetilde Q} ,  \widetilde Q; \bh \Big) 
\cdot 
G \Big( \lambda \bh \cdot \widetilde Q \frac{\partial}{\partial \widetilde Q} ,  \widetilde Q; \bh \Big) \cdot 1 = 0, \qquad \text{ in } \cM^\ab.
\end{equation}
Condition (i) now means that the image of $F \Big( \lambda \bh \cdot \widetilde Q \frac{\partial}{\partial \widetilde Q} ,  \widetilde Q; \bh \Big) $ in $\QH_\sT (X^\ab) \cong \gr \cM^\ab$ is not a zero divisor. 
As a result, $F \Big( \lambda \bh \cdot \widetilde Q \frac{\partial}{\partial \widetilde Q} ,  \widetilde Q; \bh \Big) \in \cD$ itself is not a zero divisor of $\cM^\ab$. 
The equation ~\eqref{eqn-FG-0} then implies
$$
G \Big( \lambda \bh \cdot \widetilde Q \frac{\partial}{\partial \widetilde Q} ,  \widetilde Q; \bh \Big) \cdot 1 = 0, \qquad \text{ in } \cM^\ab, 
$$
and hence the result.
The statement for $Y$ can be proved in the same way.
\end{proof}

\subsection{$\QH_\sT (X)$ via abelianization}

\begin{Lemma} \label{Lemma-shift-Q}
Suppose that for some $F(\xi, \widetilde Q; \bh) \in H_{ \sK \times \sT\times\bC^*_\bh}^* (\pt) [\![\widetilde Q]\!]$, the abelian $I$-function $I^{(F(\xi, \widetilde Q; \bh))}_{X^\ab} (\widetilde Q; \bh)$ lies in the non-localized equivariant cohomology $H^*_{\sT \times \bC^*_\bh} (X) [\![\widetilde Q]\!]$. 
Then
$$
\Big( \frac{1}{\gamma (\xi; \bh) |_{X^\ab}} I_{X^\ab}^{(\gamma(\xi; \bh) \cdot F(\xi, \widetilde Q; \bh))} (\widetilde Q; \bh) \Big) \Big|_{\bh = 0} = I^{(F(\xi, \widetilde Q; \bh))}_{X^\ab} (\widetilde Q_\sharp; \bh) \Big|_{\bh=0}, 
$$
where $\widetilde Q_\sharp$ means the shifted K\"ahler parameter $ (-1)^{ 2\rho} \cdot \widetilde Q$, i.e.
$$
\widetilde Q_\sharp^d := (-1)^{ \la 2\rho, d\ra } \cdot \widetilde Q^d, \qquad \rho := \frac{1}{2} \sum_{\alpha \in \Phi_+} \alpha. 
$$
\end{Lemma}

\begin{proof}
By the explicit formula \eqref{eqn-explicit-I}, 
\ben
\frac{1}{\gamma (\xi; \bh) |_X} I_{X^\ab}^{(\gamma (\xi; \bh) \cdot F(\xi, \widetilde Q; \bh))} (\widetilde Q; \bh) &=& \sum_{d\in \Eff(X^\ab)} \frac{\prod_{l=0}^\infty (u_i + c_1 (L_{\lambda_i}) - l\bh)}{\prod_{l=0}^\infty (u_i + c_1 (L_{\lambda_i})  + \la \lambda_i, d \ra \bh - l \bh)} \\
&&  \cdot \frac{\gamma  (\xi + d\bh; \bh) |_X}{\gamma (\xi; \bh) |_X} \cdot F (\xi + d\bh, \widetilde Q; \bh) |_X \cdot \widetilde Q^d\,.
\een
It then suffices to observe that
$$
\frac{\gamma (\xi + d\bh; \bh)}{\gamma(\xi; \bh)} = \prod_{\alpha\in \Phi} \frac{\prod_{l=0}^\infty (\alpha + \la \alpha, d \ra \bh -  l\bh)}{ \prod_{l=0}^\infty (\alpha - l\bh)} 
= \prod_{\substack{\alpha\in \Phi \\ \la \alpha, d \ra < 0 } } \frac{ 1 }{\prod_{l=0}^{- \la \alpha, d \ra -1 } (\alpha - l\bh)}
\prod_{\substack{\alpha\in \Phi \\ \la \alpha, d \ra > 0 } } \prod_{l=1}^{ \la \alpha, d \ra} (\alpha  + l \bh),
$$
whose $\bh=0$ specialization is
$$
\prod_{\substack{\alpha\in \Phi \\ \la \alpha, d \ra >0 } } (-1)^{\la \alpha, d\ra } = \prod_{\alpha\in \Phi_+ } (-1)^{\la \alpha, d\ra } ,
$$
since $(-1)^{\la \alpha, d \ra} = (-1)^{\la -\alpha, d \ra}$.
\end{proof}

We have an explicit presentation of the quasimap quantum cohomology.

\begin{Theorem} \label{Thm-QH}
We have
$$
\QH_\sT (X) \cong H_\sT^* (\pt) \otimes_\bC H_{ \sK }^*(\pt)^\sW [\![Q]\!] / \cJ,
$$
where $\cJ$ is the ideal generated by
\begin{equation} \label{eqn-QH-rel}
\frac{1}{e} \sum_{w\in \sW} (-1)^{l(w)} w \cdot \Bigg[ g(\xi) \cdot \Big(  \prod_{i: \la \lambda_i, d \ra >0} (u_i + \lambda_i )^{\la \lambda_i, d\ra} -   Q^{\bar d}_\sharp \prod_{i: \la \lambda_i, d \ra < 0} ( u_i + \lambda_i )^{- \la \lambda_i, d\ra} \Big)  \Bigg], 
\end{equation}
for all $d\in \Eff(X^\ab) \subset \sX_*(\sK)$ and $g(\xi) \in H_{ \sK \times \sT}^*(\pt)$. 
Here $Q_\sharp$ means the shifted K\"ahler parameter $ (-1)^{ 2\rho} \cdot Q$, i.e.\footnote{Note that this is well-defined, since if $\bar d = \bar c$ for $c, d\in \sX_*(\sK)$, then $c-d$ lies in the coroot lattice of $\sG$, whose pairing with $2\rho$ is in $2\bZ$.}
$$
 Q_\sharp^{\bar d} := (-1)^{ \la 2\rho, d\ra } \cdot Q^{\bar d}, \qquad \rho := \frac{1}{2} \sum_{\alpha \in \Phi_+} \alpha. 
$$
\end{Theorem}

\begin{proof}
Given $d\in \Eff(X^\ab) \subset \sX_*(\sK)$, let
$$
F_d (\xi, \widetilde Q; \bh) := g (\xi)
\cdot 
\Bigg( \prod_{i: \la \lambda_i, d \ra >0} \prod_{m = 0}^{\la \lambda_i, d \ra -1} (  u_i + \lambda_i - m\bh ) - 
\widetilde Q^d  
\prod_{i: \la \lambda_i, d \ra < 0} \prod_{m = 1}^{- \la \lambda_i, d \ra } ( u_i +  \lambda_i - m\bh )  
\Bigg)  .
$$
Prop. ~\ref{Prop-diff-QH-ab} 1) shows that 
$$
0 = \widetilde I^{(F_d (\xi, \widetilde Q; \bh))}_{X^\ab} (\widetilde Q; \bh) = F_d \Big( \lambda \widetilde Q \frac{\partial}{\partial \widetilde Q}, \widetilde Q; \bh \Big) \ \widetilde I^{(1)}_{X^\ab} (\widetilde Q; \bh).
$$
In particular, it lies in the non-localized equivariant cohomology.
Since $\widetilde I^{(1)}_{X^\ab} (\widetilde Q; \bh)$ is $\sW$-invariant (this follows from e.g \cite[Lemma ~3.5.1]{Web}), we have
$$
0 = w \cdot \widetilde I^{(F_d (\xi, \widetilde Q; \bh))}_{X^\ab} (\widetilde Q; \bh) = \Big( w \cdot F_d \Big( \lambda \widetilde Q \frac{\partial}{\partial \widetilde Q}, \widetilde Q; \bh \Big) \Big) \ \widetilde I^{(1)}_{X^\ab} (\widetilde Q; \bh) = \widetilde I^{(w \cdot F_d (\xi, \widetilde Q; \bh))}_{X^\ab} (\widetilde Q; \bh) .
$$
Summation over $\sW$ and pullback along $j$ give
$$
I^{\left(\sum_{w\in \sW} (-1)^{l(w)} w \cdot F_d (\xi, \widetilde Q; \bh) \right)}_{Y} (\widetilde Q; \bh) = j^* I^{\left(\sum_{w\in \sW} (-1)^{l(w)} w \cdot F_d (\xi, \widetilde Q; \bh) \right)}_{X^\ab} (\widetilde Q; \bh) = 0.
$$
Being $\sW$-anti-symmetric, the insertion factors into
$$
\sum_{w\in \sW} (-1)^{l(w)} w \cdot F_d (\xi, \widetilde Q; \bh) = e \cdot \Big( \frac{1}{e} \sum_{w\in \sW} (-1)^{l(w)} w \cdot F_d (\xi, \widetilde Q; \bh) \Big).
$$
On the other hand, the class $e|_Y$ is not a zero divisor in $H_\sT^*(Y)$. 
By Lemma ~\ref{Lemma-0-div}, and applying the shift $\widetilde Q \mapsto \widetilde Q_\sharp$, we have
$$
 j^* I^{\left( e^{-1} \sum_{w\in \sW} (-1)^{l(w)} w \cdot F_d (\xi, \widetilde Q_\sharp; \bh) \right)}_{X^\ab} (\widetilde Q_\sharp; \bh) = 0.
$$
Lemma ~\ref{Lemma-shift-Q} then implies 
$$
j^* \Big( \frac{1}{\gamma (\xi; \bh) |_{X^\ab}} I_{X^\ab}^{(\gamma (\xi; \bh) \cdot e^{-1} \sum_{w\in \sW} (-1)^{l(w)} w \cdot F_d (\xi, \widetilde Q_\sharp; \bh) )} (\widetilde Q; \bh) \Big) \Big|_{\bh = 0} = 0.
$$
Applying the specialization $\widetilde Q^{d} \mapsto Q^{\bar d}$ and then $(\pi^*)^{-1}$, we obtain
$$
I_X^{( e^{-1} \sum_{w\in \sW} (-1)^{l(w)} w \cdot F_d (\xi, Q_\sharp; \bh)  )} (Q, \bh) |_{\bh=0} = 0.
$$

Now let $\cJ' \subset H_\sT^*(\pt) \otimes_\bC H_\sK^*(\pt)^\sW [\![Q]\!]$ be the ideal generated by all relations of the form ~\eqref{eqn-QH-rel}. 
The above argument shows that $\cJ' \subset \cJ$, and there is a surjection of rings 
$$
\Pi: H_\sT^*(\pt) \otimes_\bC H_\sK^*(\pt)^\sW [\![Q]\!] / \cJ' \twoheadrightarrow H_\sT^*(\pt) \otimes_\bC H_\sK^*(\pt)^\sW [\![Q]\!] / \cJ.
$$
Moreover, by the explicit presentation ~\eqref{eqn-QH-rel}, we know that
\ben
\left( H_\sT^*(\pt) \otimes_\bC H_\sK^*(\pt)^\sW [\![Q]\!] / \cJ' \right) \otimes_{\bC [\![Q]\!]} \bC 
&\cong& H_\sT^*(\pt) \otimes_\bC H_\sK^*(\pt)^\sW / \im ( \cJ' \otimes_{\bC [\![Q]\!]} \bC )  \\
&=& H_\sT^*(\pt) \otimes_\bC H_\sK^*(\pt)^\sW / ( \cJ' |_{Q = 0} ) \\
&\cong&  H^*_\sT (X) ,
\een
where the last isomorphism is Lemma ~\ref{Lemma-ab-H} 3).
So $\Pi |_{Q=0}$ is the identity map on $H^*_\sT (X)$. 

On the other hand, given any $\beta \in \Eff(X)$, and $f (\xi, Q) \in H_\sT^*(\pt) \otimes_\bC H_\sK^*(\pt)^\sW [\![Q]\!]$, we have $\widehat{Q^\beta f (\xi, Q)} (Q) = Q^\beta \cdot \widehat{ f (\xi, Q)} (Q)$, and hence
$$
Q^\beta f (\xi, Q) \in \cJ  \quad \Rightarrow \quad f(\xi, Q) \in \cJ, 
$$
which implies that $H_\sT^*(\pt) \otimes_\bC H_\sK^*(\pt)^\sW [\![Q]\!] / \cJ$ is flat over $\bC[\![Q]\!]$.
Flatness of $H_\sT^*(\pt) \otimes_\bC H_\sK^*(\pt)^\sW [\![Q]\!] / \cJ$ over $ \bC [\![Q]\!]$ implies the short exact sequence
\[
\xymatrix{
0 \ar[r] & \ker \Pi \otimes_{ \bC [\![Q]\!]} \bC \ar[r] & H^*_\sT (X) \ar[r]^{\sim} & H^*_\sT (X) \ar[r] & 0.  
}
\]
Hence $\ker \Pi \otimes_{H_\sT^*(\pt) \otimes_\bC \bC [\![Q]\!]} \bC = 0$. 
On the other hand, $\sT$-equivariant Kirwan surjectivity implies the $\sT$-equivariant formality, i.e. $H^*_\sT (X)$, and hence $\QH(X)$ and $\ker \Pi$, are flat over $H^*_\sT (\pt)$.
Nakayama's lemma then implies that the restriction of $\ker \Pi = 0$ to any fiber of $\Spec H^*_\sT (\pt) \otimes_\bC \bC [\![Q]\!] \to H^*_\sT (\pt)$ vanishes.
We then conclude that $\Pi$ is an isomorphism, and $\cJ' = \cJ$.
\end{proof}

\begin{Example} \label{eg-Gr}
Consider the following quiver
\[
\xymatrix{
[n] \ar[r] & (2) \ar[r] & [m] , 
}
\]
where $(2)$ is the gauge node and $[n]$, $[m]$ are the frozen nodes. 
We refer to \S \ref{Sec-quiver} for relevant notation and definitions. 
Choose stability condition $\theta >0$. 
The quiver variety is $X \cong \Tot (\cQ^{\vee, \oplus m})$, where $\cQ$ is the tautological quotient bundle of the Grassmannian $\Gr (n-2 , n)$ pametrizing $2$-dimensional quotients $\bC^n \twoheadrightarrow \bC^2$. 
For simplicity, we choose the equivariant parameters on $[n]$ to be $u_i$, and those $[m]$ to be $v_i$. 
Then Theorem \ref{Thm-QH} says that $\QH_\sT (X) \cong \bC[\xi_1, \xi_2, u]^{S_2} [\![Q]\!] / \cJ$, where $S_2$ permutes $\xi_1$, $\xi_2$, and $\cJ$ is the ideal generated by two equations
$$
\frac{
\prod_{i=1}^n (\xi_1 - u_i) - \prod_{i=1}^n (\xi_2 - u_i)
}{
\xi_1 - \xi_2
} = (-Q) \cdot \frac{
\prod_{i=1}^m (v_i - \xi_1) - \prod_{i=1}^m (v_i - \xi_2)
}{
\xi_1 - \xi_2
},
$$
$$
\frac{
\xi_1 \prod_{i=1}^n (\xi_1 - u_i) - \xi_2 \prod_{i=1}^n (\xi_2 - u_i)
}{
\xi_1 - \xi_2
} = (-Q) \cdot \frac{
\xi_1 \prod_{i=1}^m (v_i - \xi_1) - \xi_2 \prod_{i=1}^m (v_i - \xi_2)
}{
\xi_1 - \xi_2
}.
$$
\end{Example}

\subsection{Finiteness and variation of GIT}\label{sec:variationofGIT}

In this subsection, we observe that the quantum relations in Thm. ~\ref{Thm-QH} exhibit certain symmetries which ``invert" the K\"ahler parameters. 
This will actually lead to an invariance property of the quasimap quantum cohomology under a variation of GIT. 

Let $X$ be as above. 
For the purpose of inverting the K\"ahler parameters, its convenient to restrict the quasimap quantum cohomology allow only \emph{polynomials} in $Q$. 
According to Thm. ~\ref{Thm-QH}, we define the \emph{polynomial quasimap quantum cohomology} as\footnote{As in Rem. ~\ref{Rk-power-series}, $(-) [Q]$ here precisely means $(-) [Q^{\Eff(X)} ]$.}
$$
\QH_\sT^\poly (X) :=  H^*_\sT (\pt) \otimes_\bC H_\sK^*(\pt)^\sW [Q] / \cJ_\poly, 
$$
where we $\cJ_\poly$ means the ideal in $H^*_\sT (\pt) \otimes_\bC H_\sK^*(\pt)^\sW [Q]$ generated by the same relations as in ~\eqref{eqn-QH-rel}. 

\begin{Corollary}[Finiteness]
\label{Cor-finiteness}

The polynomial quasimap quantum cohomology is a subalgebra
$$
\QH_\sT^\poly (X) \hookrightarrow \QH_\sT^\poly (X) \otimes_{\bC[Q]} \bC[\![Q]\!] \cong \QH_\sT (X). 
$$
In particular, let $\{ \tau_i (\xi), 1\leq i \leq N \} \subset H^*_\sT (\pt) \otimes_\bC H^*_\sK (\pt)^\sW$ be a collection of elements, such that $\{\tau_i (\xi)|_X\}$ forms a basis of $H^*_\sT (X)$ over $H^*_\sT (\pt)$. 
Then the structure constants of 
$$
\widehat{ \tau_i (\xi)} (Q) * \widehat{ \tau_j (\xi)} (Q)
$$
are polynomials in $Q^\beta$, $\beta \in \Eff(X)$.
\end{Corollary}

\begin{proof}
The embedding of $\QH_\sT^\poly (X)$ into $\QH_\sT (X)$ follows from the Noetherian-ness of $\bC[Q]$.
Let $\{ \tau_i (\xi), 1\leq i \leq N \} \subset H^*_\sT (\pt) \otimes_\bC H^*_\sK (\pt)^\sW$ be a collection of elements as above. 
The map $f:  \widehat{\tau_i (\xi)} (Q) \mapsto \tau_i (\xi)$ gives an isomorphism of free $\bC[\![Q]\!]$-modules $H^*_\sT (X) [\![Q]\!] \cong H^*_\sT (\pt) \otimes_\bC H^*_\sK (\pt)^\sW [\![Q]\!] / \cJ $,
which is also an isomorphism of rings, where the LHS is endowed with the quasimap quantum product. 
This isomorphism fits into the diagram 
\[
\xymatrix{
 \bigoplus_{i=1}^N \widehat{\tau_i (\xi)} (Q) \cdot ( H^*_\sT (\pt) \otimes_\bC \bC[Q] ) \ar@{^(->}[d]^h \ar[r]^-{f_\poly} & H^*_\sT (\pt) \otimes_\bC H^*_\sK (\pt)^\sW [Q] / \cJ_\poly  \ar@{^(->}[d]_g \\
 \bigoplus_{i=1}^N \widehat{\tau_i (\xi)} (Q) \cdot ( H^*_\sT (\pt) \otimes_\bC \bC[\![Q]\!] ) \ar[r]^-\sim_-f & H^*_\sT (\pt) \otimes_\bC H^*_\sK (\pt)^\sW [\![Q]\!] / \cJ
}
\]
where $f$ and $g$ are also maps of algebras.
Now since $\bC[Q] \hookrightarrow \bC[\![Q]\!]$ is faithfully flat, $f$ being an isomorphism implies that $f_\poly$ is also an isomorphism. 
We then equip $ \bigoplus_{i=1}^N \widehat{\tau_i (\xi)} (Q) \cdot ( H^*_\sT (\pt) \otimes_\bC \bC[Q]) $ with the ring structure induced from $f_\poly$. 
The diagram becomes a commutative diagram of algebras, and the statement follows.
\end{proof}

%
%

On the other hand, we have an embedding of algebras
$$
\QH_\poly (X) \hookrightarrow \QH_\poly (X) \otimes_{\bC[Q]} \bC [Q^{\pm 1}] \cong H_\sK^*(\pt)^\sW [Q^{\pm 1}] / ( \cJ_X \otimes_{\bC[Q]} \bC[Q^{\pm 1}] ), 
$$
where $\bC[Q^{\pm 1}] := \bC[Q^{\pi_1 (\sG)}]$, and the injectivity follows from the fact that $Q$ is not a zero-divisor in $\QH(X)$, and hence in $\QH_\poly (X)$. 
The following result shows that upon the base change to $\bC[Q^{\pm 1}]$, the polynomial quasimap quantum cohomology  \emph{independent} of the choice of the generic stability condition $\theta$.

\begin{Corollary}[Variation of GIT]
\label{Cor-vGIT}

Let $\theta'$ be another generic stability condition, such that the $\sG$-action on $\sN^{\theta\text{-s}}$ is free.
Let $X'$ be the GIT quotient. 
Then
$$
\QH_\poly (X) \otimes_{\bC[Q]} \bC [Q^{\pm 1}] \ \cong \ \QH_\poly (X') \otimes_{\bC[Q]} \bC [Q^{\pm 1}]
$$
as algebras.
\end{Corollary}

\begin{proof}
Let $d\in \Eff(X^\ab)$. 
The ideal $\cJ_X \otimes_{\bC[Q]} \bC[Q^{\pm 1}]$ is generated by the elements
$$
\frac{1}{e} \sum_{w\in \sW} (-1)^{l(w)} w \cdot \Bigg[ g(\xi) \cdot \Big(  \prod_{i: \, \la \lambda_i, d \ra >0} ( u_i + \lambda_i)^{\la \lambda_i, d \ra} -   Q^{\bar d}_\sharp \prod_{i: \, \la \lambda_i, d \ra < 0} (u_i + \lambda_i)^{- \la \lambda_i, d \ra} \Big)  \Bigg].
$$
It follows that the elements
$$
\frac{1}{e} \sum_{w\in \sW} (-1)^{l(w)} w \cdot \Bigg[ g(\xi) \cdot \Big( - Q^{-\bar d}_\sharp \prod_{i: \, \la \lambda_i, d \ra >0} ( u_i + \lambda_i)^{\la \lambda_i, d \ra}  +   \prod_{i: \, \la \lambda_i, d \ra < 0} (u_i + \lambda_i)^{- \la \lambda_i, d \ra} \Big)  \Big)  \Bigg],  
$$
which is exactly the relation given by $-d$, also lies in  $\cJ_X \otimes_{\bC[Q]} \bC[Q^{\pm 1}]$.
In other words, $\cJ_X \otimes_{\bC[Q]} \bC[Q^{\pm 1}]$ can be rephrased as generated by the same relations \emph{for all} $d \in \sX_* (\sK)$. 
This implies $\cJ_{X'} \otimes_{\bC[Q]} \bC[Q^{\pm 1}] = \cJ_X \otimes_{\bC[Q]} \bC[Q^{\pm 1}]$, which proves the result.
\end{proof}

\begin{Example}
In Example \ref{eg-Gr}, if we choose the opposite stability condition $\theta <0$, then the quiver variety is $X' \cong \Tot (\cS^{\oplus n})$, where $\cS$ is the tautological subbundle of the Grassmannian $\Gr (2, m)$ parametrizing $2$-dimensional subspaces $\bC^2 \subset \bC^m$. 
$\QH_\sT (X')$ then has two relations
$$
(-Q)^{-1} \cdot \frac{
\prod_{i=1}^n (\xi_1 - u_i) - \prod_{i=1}^n (\xi_2 - u_i)
}{
\xi_1 - \xi_2
} =  \frac{
\prod_{i=1}^m (v_i - \xi_1) - \prod_{i=1}^m (v_i - \xi_2)
}{
\xi_1 - \xi_2
},
$$
$$
(-Q)^{-1} \cdot \frac{
\xi_1 \prod_{i=1}^n (\xi_1 - u_i) - \xi_2 \prod_{i=1}^n (\xi_2 - u_i)
}{
\xi_1 - \xi_2
} =  \frac{
\xi_1 \prod_{i=1}^m (v_i - \xi_1) - \xi_2 \prod_{i=1}^m (v_i - \xi_2)
}{
\xi_1 - \xi_2
},
$$
which coincide with the equations in Example \ref{eg-Gr} after inverting $Q$.
\end{Example}


\subsection{Asymptotics and rigidity}

In this subsection, we prove some rigidity results that, under certain conditions, the quantum tautological classes $\widehat{\tau(\xi)} (Q)$ may coincide with their classical restrictions $\tau(\xi) |_X$.

Let $H^*_\sK (\pt)$ be equipped with the standard grading, where
$$
\deg \lambda =1, \qquad \text{for any } \lambda \neq 0 \in \sX^* (\sK) .
$$
For any $\tau(\xi) \in H_\sK^*(\pt)$, and  $d\in \sX_* (\sK)$, denote
$$
\deg_d \tau(\xi) := \deg_\bh \tau(\xi + d\bh),
$$ 
where $\deg_\bh$ means the $\bh$-degree of the polynomial $\tau(\xi + d\bh)$.

\begin{Lemma} \label{Lemma-asymp}
Let $\tau(\xi) \in H_\sT^* (\pt) \otimes_\bC H^*_\sK(\pt)^\sW$. 
\begin{enumerate}[label=\arabic*)]
\setlength{\parskip}{1ex}

\item If $\deg_d \tau(\xi) \leq  \la \det \sN, d \ra$ for any $d \in \Eff(X^\ab)$, then the Laurent expansion of $ I^{(\tau(\xi))}_{X} (Q; \bh) $ lies in $H^* (X) [\![\bh^{-1}]\!] [\![Q]\!]$, i.e. the limit
$
\lim_{\bh \to \infty} I^{(\tau(\xi))}_{X} ( Q; \bh) 
$
exists. 

\item If $\deg_d \tau(\xi) <  \la \det \sN, d \ra$ for any $d \in \Eff(X^\ab)$, then
$$
\lim_{\bh \to \infty} I^{(\tau(\xi))}_{X} ( Q; \bh) = \tau(\xi) |_X. 
$$
\end{enumerate}

\begin{proof}
Recall that we have an explicit formula for the $I$-function with insertion $\tau(\xi) \in H^*_\sK (\pt)^\sW$, obtained by \eqref{eqn-explicit-I} and Prop. ~\ref{Prop-abelianization-I}:
\begin{eqnarray}
 I^{(\tau(\xi))}_X (Q; \bh) &=& \sum_{d\in \Eff(X^\ab)} Q^{\bar d} \cdot \tau(\xi + d\bh) |_X \cdot \prod_{i=1}^n \frac{\prod_{l=0}^\infty (u_i + c_1 (L_{\lambda_i}) - l\bh) 
}{
\prod_{l=0}^\infty (u_i + c_1 (L_{\lambda_i})  + \la \lambda_i, d \ra \bh - l \bh)
}     \nonumber 	\\
&& \cdot \prod_{\alpha \in \Phi}  
\frac{\prod_{l=0}^\infty (c_1 (L_\alpha) + \la \alpha, d \ra \bh - l\bh) 
}{
\prod_{l=0}^\infty (c_1 (L_\alpha )   - l \bh)
}
\end{eqnarray}
As $\bh \to \infty$, the $Q^{\bar d}$-coefficient of the $I$-function has $\bh$-degree
$$
\deg_d \tau(\xi) -  \sum_{i=1}^n \la \lambda_i, d \ra = \deg_d \tau (\xi) -  \la \det \sN, d \ra. 
$$
Hence 1). 
For 2), when the inequality is strict, only the constant term $\tau(\xi)|_X$ survives under the limit $\bh \to \infty$. 
\end{proof}

\end{Lemma}

\begin{Proposition} \label{Prop-rigidity}
Let $\tau(\xi) \in H_\sT^* (\pt) \otimes_\bC H^*_\sK(\pt)^\sW$. 
\begin{enumerate}[label=\arabic*)]
\setlength{\parskip}{1ex}

\item If $\deg_d \tau(\xi) \leq  \la \det \sN, d \ra$ for any $d \in \Eff(X^\ab)$, then the capped $I$-functions with insertions coincide with the quantum tautological classes, i.e.
$$
\widehat I^{(\tau (\xi) )} (Q; \bh) = \widehat{\tau (\xi)} (Q) .
$$

\item If $\deg_d \tau(\xi) <  \la \det \sN , d \ra$ for any $d \in \Eff(X^\ab)$, then both of them are ``rigid", i.e.
$$
\widehat I^{(\tau (\xi) )} (Q; \bh) = \widehat{\tau (\xi)} (Q) = \tau(\xi) |_X .
$$

\end{enumerate}

\end{Proposition}

\begin{proof}
1): Recall Prop.  ~\ref{Prop-Birkhoff} that 
$$
\widehat I^{(\tau(\xi))} (Q; \bh) = \bS^*(Q; -\bh, 0) I^{(\tau (\xi))} (Q; \bh), 
$$
whose LHS lies in the $\bC^*_\bh$-non-localized cohomology $H^*_{\sT \times \bC^*_\bh} (X) [\![Q]\!]$, while both factors of the RHS a priori lie in the localized cohomology $H^*_{\sT \times \bC^*_\bh} (X)_\loc [\![Q]\!]$. 
However, the fact that $\bS^*(Q; -\bh, 0) = \Id + O(\bh^{-1})$ and Lemma \ref{Lemma-asymp} 1) imply that the RHS is bounded as $\bh \to \infty$. 
Therefore, $\lim_{\bh \to \infty} \widehat I^{(\tau(\xi))} (Q; \bh)$ exists and hence must be independent of $\bh$. 
It is then equal to the specialization at $\bh=0$, i.e. $\widehat{\tau (\xi)} (Q)$. 

2): Being independent of $\bh$, one can also compute $\widehat I^{(\tau(\xi))} (Q; \bh)$ by letting $\bh \to \infty$, which is $\tau(\xi) |_X$ according to  Lemma ~\ref{Lemma-asymp} 2).
\end{proof}

\vspace{4ex}

\section{Quasimap quantum cohomology of quiver varieties}

\subsection{Quiver varieties} \label{Sec-quiver}

A \emph{quiver} is a finite oriented graph, denoted by $\bfQ = (\bfQ_0 = \bfI \sqcup \bfF, \bfQ_1)$, satisfying the following.
\begin{itemize}
\setlength{\parskip}{1ex}

\item $\bfQ_0 = \bfI \sqcup \bfF$ is the set of nodes, where $\bfI$ is the set of \emph{gauge nodes}, and $\bfF$ is the set of \emph{frozen nodes}.

\item  $\bfQ_1$ is the set of edges, with \emph{source and target} maps $s, t: \bfQ_1 \to \bfQ_0$. 

\item There are no self loops or oriented 2-cycles.
\end{itemize}

A \emph{dimension vector} is a vector $\bv = (\bv_i)_{i\in \bfQ_0} \in \bZ^{\bfQ_0}$. 
Given a dimension vector, one can form a vector space of quiver representations
$$
\sN := \bigoplus_{e\in \bfQ_1} \Hom (V_{s(e)}, V_{t(e)}), 
$$
where $V_i \cong \bC^{\bv_i}$ for $i\in \bfQ_0$. 
The \emph{gauge group} and \emph{flavor group}
$$
\sG := \prod_{i\in \bfI} GL (\bv_i), \qquad \sG_F := \prod_{i\in \bfF} GL (\bv_i)
$$
act naturally on $\sN$. 
More precisely, for a point $B = (B_e)_{e\in \bfQ_1} \in \sN$, an element $g = (g_i)_{i\in \bfQ_0} \in \sG \times \sG_F$ acts by
$$
( g \cdot B )_e = g_{t(e)} B_e g_{s(e)}^{-1} .
$$
A \emph{stability condition} is a vector $\theta = (\theta_i)_{i\in \bfI} \in \bZ^{\bfI}$, which determines a character
$$
\chi_\theta: \sG \to \bC^*, \qquad g = (g_i)_{i\in \bfI} \mapsto \prod_{i\in \bfI}  ( \det g_i )^{\theta_i}. 
$$
The GIT stability is equivalent\footnote{This also requires the trick of Crawley-Boevey \cite{CB} which relates framed and unframed quivers.} to the stability for quiver representations \cite{Kin}, where it is more convenient to take $\theta \in \bR^{\bfI}$. 
For generic choices of $\theta$, one has $\sN^{\sG\text{-ss}} = \sN^{\sG\text{-s}}$. 
Moreover, the $\sG$-action on $\sN^{\sG\text{-s}}$ is free, by \cite[Lemma ~1.3.2]{Gin}.

Let $\sT \subset \sG_F$ be the maximal torus, i.e. $\sT = \prod_{i\in \bfF} (\bC^*)^{\bv_i}$. 
Then $\sT$ is a flavor torus acting on $\sN$, which commutes with $\sG$.

\begin{Definition}
Given a quiver $\bfQ = (\bfQ_0 = \bfI \sqcup \bfF, \bfQ_1)$, with dimension vector $\bv$ and stability condition $\theta$, the \emph{quiver variety} is defined as the GIT quotient
$$
X := \sN /\!\!/_\theta \sG = \sN^{\sG\text{-s}} / \sG.
$$
\end{Definition}

\begin{Lemma} \label{Lemma-stab}
Fix a gauge node $k \in \bfI$, and let $(B_e, V_i)_{e\in \bfQ_1, i\in \bfQ_0}$ be in $\sN^{\sG\text{-s}}$. 

\begin{enumerate}[label=\arabic*)]
\setlength{\parskip}{1ex}

\item If $\theta_k >0$, then the map
$$
\sum_{e \in \bfQ_1,  \, t(e) = k  } B_e: \  \bigoplus_{e \in \bfQ_1, \, t(e) = k} V_{s(e)} \to V_k
$$
is surjective. 

\item If $\theta_k <0$, then the map
$$
\bigoplus_{e \in \bfQ_1,  \, s(e) = k  } B_e: \ V_k \to \bigoplus_{e \in \bfQ_1 , \, s(e) = k  } V_{t(e)}
$$
is injective. 

\end{enumerate}

\end{Lemma}

\begin{proof}
We only prove 1); the proof of 2) is similar.
Applying the construction of Crawley-Boevey ~\cite{CB}, we may replace all frozen nodes by arrows from/to an extra node $\infty$.
Therefore, we may assume that $\bfF = \emptyset$ and $\theta \cdot \bv = 0$. 
Suppose that the map $B := \sum_{e \in \bfQ_1,  \, t(e) = k  } B_e$ fails to be surjective. 
Then $\im B \neq V_k$. 
Denote by $V''$ the quiver representation which is $\im B$ at the node $k$, and $V_i$ for $i\neq k$. 
Then $V''$ is a subrepresentation of $V$, with $\theta \cdot \dim (V / V'') =  \theta_k \cdot \dim \im B >0 = \theta \cdot \bv$; a contradiction to the $\theta$-stability of $V$.
\end{proof}

Given our definition of $\sG$, we have 
$$
\sK = \prod_{k \in \bfI} (\bC^*)^{\bv_k }, \qquad \sX_* (\sK) = \prod_{k \in \bfI} \bZ^{\bv_k }, \qquad \pi_1 (\sG) =  \bZ^\bfI, \qquad \sW = \prod_{k\in \bfI} S_{\bv_k}.
$$ 
An element of $\sX_* (\sK)$ is a tuple of integers
$$
d = (d_j^{(k )})_{k \in \bfI, \, 1\leq j\leq \bv_k}, 
$$
and the map of degrees is given by
$$
\sX_* (\sK) \to \pi_1 (\sG), \qquad d \mapsto \bar d = (\bar d^{(k )})_{k \in \bfI}, \qquad \text{where} \quad \bar d^{(k )} :=  \sum_{j=1}^{\bv_k } d_j^{(k )} .
$$
Let $e_j^{(k)}$, $\xi_j^{(k)}$, $k\in \bfI$, $1\leq j\leq \bv_k$ be the standard bases of $\sX_*(\sK)$ and $\sX^*(\sK)$ respectively. 
In other words, $d = \sum_{k\in \bfI} \sum_{j=1}^{\bv_k} d_j^{(k)} e_j^{(k)}$. 

It is clear that the following coweights lie in the abelian effective cone
$$
\sign (\theta_k) \cdot e_j^{(k)} \in \Eff(X^\ab), \qquad k \in \bfI, \ 1\leq j\leq \bv_k,
$$ 
where we denote $\sign (\theta_k) := 1$ (resp. $-1$), if $\theta_k >0$ (resp. $\theta_k <0$).

We denote the abelian and nonabelian K\"ahler parameters by (note that $\overline{e^{(k)}_1} = \overline{e^{(k)}_j}$ for any $j$)
$$
\widetilde Q_j^{(k)} := \widetilde Q^{e_j^{(k)}}, \qquad Q^{(k)} := Q^{\overline{e^{(k)}_1 }}
$$
and the specialization map $\widetilde Q \mapsto Q$ is given by
$$
\widetilde Q_j^{(k)} \mapsto Q^{(k)}.  
$$
Similarly, denote the equivariant parameters for $\sT$ by $u_{j'}^{(k')}$, for those $k' \in \bfF$ and $1\leq j' \leq \bv_{k'}$.

The decomposition of $\sN$ as a $(\sK \times \sT)$-representation is
$$
\sN = \sum_{e\in \bfQ_1} \sum_{a=1}^{\bv_{s(e)}} \sum_{b =1}^{\bv_{t(e)}} \bC_{\xi^{(t(e))}_b - \xi^{(s(e))}_a } ,
$$
where we mean $\xi_{j'}^{(k')} = u_{j'}^{(k')}$, for those $k'\in \bfF$.

\subsection{Quantum multiplication relations among Chern classes}\label{sec:Chernclassmult}

To save notation, in this subsection, we fix a node $k\in \bfI$ and denote 
$$
V := V_k, \qquad \bv := \bv_k, \qquad \widetilde Q_j := \widetilde Q_j^{(k)}, \qquad Q := Q^{(k)}, \qquad \xi_j := \xi_j^{(k)},
$$
$$
V_- := \bigoplus_{e\in \bfQ_1, \, t(e) = k} V_{s(e)}, \qquad V_+ := \bigoplus_{e\in \bfQ_1, \, s(e) = k} V_{t(e)}, \qquad \bv_\pm := \dim V_\pm , 
$$
$$
\mu := \{ \mu_a \mid 1 \leq a \leq \bv_- \} := \{ \xi^{(s(e))}_a \mid e \in \bfQ_1, \, t(e) = k, \, 1\leq a\leq \bv_{s(e)} \}, \quad \text{i.e. } V_- = \sum_{a=1}^{\bv_-} \bC_{\mu_a}
$$
$$
\nu := \{ \nu_b \mid 1 \leq b \leq \bv_+ \} := \{ \xi^{(t(e))}_b \mid e \in \bfQ_1, \, s(e) = k, \, 1\leq b\leq \bv_{t(e)} \}, \quad \text{i.e. } V_+ = \sum_{b=1}^{\bv_+} \bC_{\nu_b}. 
$$
Again, we set $\xi_{j'}^{(k')} = u_{j'}^{(k')}$ if $k' \in \bfF$. 

Let $t$ be a formal variable, and we consider the class
$$
c_t (V) := \prod_{j=1}^{\bv} (t + \xi_j) = \sum_{i=1}^\bv t^i e_{\bv-i} (\xi),  
$$
and similarly $c_t(V_\pm)$. 
We will derive a ``quantum division" relation among $c_t(V)$ and $c_t(V_\pm)$. 

\begin{Remark}
A priori, each $t$-coefficient of $c_t (V)$, i.e. $e_i (\xi)$, gives an element in $H_\sT^*(\pt) \otimes_\bC H_\sK^*(\pt)^\sW$. 
In the following, we will denote its image in $\QH(X) = H_\sT^*(\pt) \otimes_\bC H_\sK^*(\pt)^\sW[\![Q]\!] / \cJ$ also by $e_i (V)$. 
Note that in this definition, all multiplications involved in $e_i(\xi)$ are quantum multiplications.
One should not confuse this with the tautological class $e_i (\cV)$ (which is $e_i (\xi) |_X$ for us) in $H^*_\sT (X)$. 
An alternative way to understand this distinction is via the quantum tautological map \eqref{eqn-quantum-taut} -- in the presentation $(H^*_\sT (X) [\![Q]\!], *)$, the classes here are $\widehat{e_i (\xi)} (Q)$. 
\end{Remark}

According to Thm. ~\ref{Thm-QH}, relations of the quasimap quantum cohomology $\QH(X)$ are of the form ~\eqref{eqn-QH-rel}, obtained by alternating summation of \emph{abelian relations} in $\QH(X^\ab)$.
It suffices to consider $d = \sign (\theta_k) \cdot e_1^{(k)}$, for all $k\in \bfI$. 

We first concentrate on the case when $\theta_k >0$. 
The abelian relation in \eqref{eqn-QH-rel-ab} corresponding to $d = e_1^{(k)}$ is 
\begin{equation} \label{eqn-rel-k}
 \prod_{a=1}^{\bv_-} (\xi_1 - \mu_a ) = \widetilde Q_1 \cdot  \prod_{b=1}^{\bv_+} (\nu_b - \xi_1 ) .
\end{equation}
Expand it in $\xi_1$:
$$
\sum_{m=0}^{\bv_-}  (-1)^{\bv_- - m} e_{\bv_- - m} (\mu) \cdot \xi_1^m = \widetilde Q_1 \cdot \sum_{m=0}^{\bv_+} (-1)^m e_{\bv_+ - m} (\nu) \cdot  \xi_1^m, 
$$
where $e_i (x_1, \cdots, x_N) = \sum_{1\leq k_1 < \cdots < k_i \leq N} x_{k_1} \cdots x_{k_i}$ is the $i$-th elementary symmetric function.

Multiply it with
$$
g(\xi) = \xi_1^p \xi_2^{\bv -2} \cdots \xi_\bv^0, \qquad p \geq 0,
$$
and apply the specialization 
$$
\widetilde Q_1 \ \mapsto \ Q_\sharp = (-1)^{\bv-1} Q, 
$$ 
together with the anti-symmetrization by $\sW$.
Note that in the Weyl group $\sW = \prod_{i \in \bfI} S_{\bv_i}$, only the $k$-th component $S_{\bv} = S_{\bv_k}$ acts nontrivially. 
We obtain in $\QH(X)$:
\ben
&& \sum_{m=0}^{\bv_-} (-1)^{\bv_- -m} e_{\bv_- -m} (\mu) \cdot \frac{
\sum_{\sigma \in S_\bv}  \sign (\sigma) \cdot \xi_{\sigma(1)}^{m+p} \cdots \xi_{\sigma(\bv)}^{0}
}{
\sum_{\sigma \in S_\bv}  \sign (\sigma) \cdot \xi_{\sigma(1)}^{\bv-1} \cdots \xi_{\sigma(\bv)}^{0}
}  \\
&=& 
(-1)^{\bv-1} Q \cdot \sum_{m=0}^{\bv_+} (-1)^m  e_{\bv_+ - m} (\nu) \cdot 
\frac{
\sum_{\sigma \in S_\bv}  \sign (\sigma) \cdot \xi_{\sigma(1)}^{m+p} \cdots \xi_{\sigma(\bv)}^{0}
}{
\sum_{\sigma \in S_\bv}  \sign (\sigma) \cdot \xi_{\sigma(1)}^{\bv-1} \cdots \xi_{\sigma(\bv)}^{0}
} 
\een
i.e.
\begin{equation} \label{eqn rel after anti-sym}
 \sum_{m=0}^{\bv_-} (-1)^{\bv_- -m} e_{\bv_- -m} (\mu) \cdot h_{m+p-\bv+1} (\xi) = (-1)^{\bv-1} Q \cdot \sum_{m=0}^{\bv_+} (-1)^m  e_{\bv_+ - m} (\nu) \cdot h_{m+p-\bv+1} (\xi), 
\end{equation}
where $h_i (x_1, \cdots, x_N) = \sum_{1\leq k_1 \leq \cdots \leq k_i \leq N} x_{k_1} \cdots x_{k_i}$ is the $i$-th complete symmetric function, and we set $h_i = 0$ for $i<0$. 

\begin{Remark}
When $\theta_k <0$, the abelian relation in \eqref{eqn-QH-rel-ab} corresponding to $d = -e_1^{(k)}$ happens to be \eqref{eqn-rel-k} with $\widetilde Q_1$ formally moved to the other side:
$$
 \prod_{a=1}^{\bv_-} (\xi_1 - \mu_a ) = \widetilde Q_1 \cdot  \prod_{b=1}^{\bv_+} (\nu_b - \xi_1 ) .
$$
The same arguments hold and \eqref{eqn rel after anti-sym} becomes
\begin{equation} \label{eqn rel after anti-sym - neg}
(-1)^{\bv-1} Q^{-1} \cdot  \sum_{m=0}^{\bv_-} (-1)^{\bv_- -m} e_{\bv_- -m} (\mu) \cdot h_{m+p-\bv+1} (\xi) =  \sum_{m=0}^{\bv_+} (-1)^m  e_{\bv_+ - m} (\nu) \cdot h_{m+p-\bv+1} (\xi)\,. 
\end{equation}
\end{Remark}

For any Laurent series $f(t) = \sum_{n=-\infty}^N a_n t^n \in H_\sK^*(\pt)^\sW (\!(t^{-1} )\!)$, we denote its positive and negative truncations by
\[
[f(t)]_+ := \sum_{n=0}^{N} a_n t^n, \qquad [f(t)]_- := \sum_{n=-\infty}^{-1} a_n t^n, \qquad \text{if } N \geq 0, 
\]
and if $N<0$, we simply let $[f(t)]_+ = 0$ and $[f(t)]_- = f(t)$.

\begin{Definition}\label{eqn:chernquot}
For any $U = \bC_{v_1} + \cdots + \bC_{v_r}$ and $U' = \bC_{v'_1} + \cdots + \bC_{v'_s}$, introduce the \emph{truncated Chern quotient} :
\ben
\delta_t (U ,  U') &:=& \Big[ \frac{c_t (U)}{c_t (U')} \Big]_+ \\
&=& \left\{ \begin{aligned}
& t^{r - s } \sum_{p=0}^{r - s} (-t)^{-p} \sum_{m=0}^{p} (-1)^{m} e_{m}(v) h_{p-m} (v' ) \quad \in \quad H_\sT^*(\pt) \otimes_\bC H_\sK^*(\pt)^\sW [t^{\pm 1}] , && \quad r\geq s \\
& 0 , && \quad r<s.
\end{aligned} \right.
\een
\end{Definition}

\begin{Theorem} \label{Thm-c_t-rel}
The following equations hold in $\QH_\sT (X) [t]$.
\begin{enumerate}[label=\arabic*)]
\setlength{\parskip}{1ex}

\item If $\theta_k >0$, then $\Big[ \dfrac{c_t (V_-)}{c_t (V)} \Big]_- = (-1)^{\bv_- -\bv+1} Q \cdot \Big[ \dfrac{c_t (V_+)}{c_t (V)} \Big]_-$, or equivalently
$$
c_t ( V_- )  - \delta_t (V_- , V) \cdot c_t (V)  =  (-1)^{\bv_- -\bv+1} Q \cdot \left( c_t ( V_+ )
- \delta_t (V_+ ,  V)  \cdot c_t (V) \right). 
$$
In particular, the classical limit of $\delta_t (V_- , V)$ in $H^*_\sT (X)$ is
$$
\delta_t (V_- , V) \big|_{Q=0} = c_t (V_- ) c_t ( V)^{-1} |_X = c_t (\cV_- / \cV) , 
$$
where $\cV = V|_X$ and $\cV_- = V_- |_X$ are the tautological vector bundles on $X$.

\item If $\theta_k <0$, then $(-1)^{\bv_- -\bv+1} Q^{-1} \cdot \Big[ \dfrac{c_t (V_-)}{c_t (V)} \Big]_- =  \Big[ \dfrac{c_t (V_+)}{c_t (V)} \Big]_-$, or equivalently
$$
(-1)^{\bv_- -\bv+1} Q^{-1} \cdot \left(  c_t ( V_- )  - \delta_t (V_- , V) \cdot c_t (V) \right)  =  c_t ( V_+ )
- \delta_t (V_+ , V)  \cdot c_t (V) . 
$$
In particular, the classical limit of $\delta_t (V_+ , V)$ in $H^*_\sT (X)$ is
$$
\delta_t (V_+ , V) \big|_{Q=\infty} = c_t (V_+ ) c_t ( V)^{-1} |_X = c_t (\cV_+ / \cV) .
$$

\end{enumerate}

\end{Theorem}

\begin{proof}
1): We apply \eqref{eqn rel after anti-sym} to compute the following quantum division class in $\QH_\sT (X) [t^{\pm 1}]$.
Note that classes like $c_t (V)$ are all invertible, and we have $\bv_- \geq \bv$ by Lemma ~\ref{Lemma-stab}. 
\ben
c_t ( V_- ) \cdot c_t (V)^{-1}
&=& 
 \prod_{a=1}^{\bv_-} (t + \mu_a ) 
\cdot \prod_{j=1}^{\bv} (t + \xi_j )^{-1}  \\
&=& 
t^{\bv_-} \prod_{a=1}^{\bv_-} (1+t^{-1} \mu_a) \cdot 
t^{-\bv} \prod_{j=1}^\bv (1 + t^{-1} \xi_j )^{-1}  \\
&=&  t^{\bv_- - \bv} \left( \sum_{m=0}^{\bv_-}  e_{m} (\mu) t^{-m} \right) \left( \sum_{m' =0}^\infty h_{m'} (\xi) (-t)^{-m'} \right) \\
&=& t^{\bv_- - \bv} \sum_{p=0}^\infty (-t)^{-p} \sum_{\substack{m+m' = p \\ 0\leq m\leq \bv_-, \, m' \geq 0}} (-1)^{m} e_{m}(\mu) h_{m'} (\xi) \\
&=&  t^{\bv_- - \bv } \sum_{p=0}^\infty (-t)^{-p} \sum_{m=0}^{\bv_-} (-1)^{m} e_{m}(\mu) h_{p-m} (\xi) .
\een
Since $h_N (\xi) = 0$ for all $N<0$, the summation $\sum_{m=0}^{\bv_-}$ is actually $\sum_{m=0}^{\min\{\bv_-, p\} }$; however, for convenience of later computations, we prefer to keep the upper bound $\bv_-$. 
Spitting the summation $\sum_{p=0}^\infty = \sum_{p=0}^{\bv_- - \bv} + \sum_{p=\bv_- - \bv +1}^\infty$, we then have
\begin{eqnarray}
 c_t ( V_- )
\cdot
 c_t (V)^{-1} 
&=& t^{\bv_- - \bv } \sum_{p=0}^{\bv_- - \bv} (-t)^{-p} \sum_{m=0}^{\bv_-} (-1)^{m} e_{m}(\mu) h_{p-m} (\xi )   \nonumber \\
&&
+ 
t^{\bv_- - \bv } \sum_{p=\bv_- -\bv +1}^\infty (-t)^{-p} \sum_{m=0}^{\bv_-} (-1)^{\bv_- - m} e_{\bv_- - m}(\mu) h_{p-\bv_- +m} (\xi)      \nonumber \\
&=& \delta_t (V_- , V) + t^{\bv_- - \bv } \sum_{p=0}^\infty (-t)^{-p-\bv_- + \bv -1 } \sum_{m=0}^{\bv_-} (-1)^{\bv_- - m} e_{\bv_- - m}(\mu) h_{p -\bv +1 +m} (\xi)      \nonumber    \\
&=& \delta_t (V_- , V)  + (-1)^{\bv_- - \bv +1} t^{-1 } \sum_{p=0}^\infty (-t)^{-p } \sum_{m=0}^{\bv_-} (-1)^{\bv_- - m} e_{\bv_- - m}(\mu) h_{p -\bv +1 +m} (\xi)     \label{eqn-V_-}.
\end{eqnarray}
Applying \eqref{eqn rel after anti-sym} to the second term,  we obtain
\begin{equation} \label{eqn-quotient-V_-}
 c_t ( V_- )
\cdot
 c_t (V)^{-1} = \delta_t (V_- , V)  + (-1)^{\bv_-  } Q \cdot \left[ t^{ - 1 }   \sum_{p=0}^\infty (-t)^{-p}   \sum_{m=0}^{\bv_+} (-1)^m e_{\bv_+ - m} (\nu) h_{p - \bv +1 +m} (\xi) \right]. 
\end{equation}
On the other hand, the computations above until \eqref{eqn-V_-} still hold\footnote{One has to treat the case $\bv_+ < \bv$ separately.} if we replace $V_-$ with $V_+$. 
Therefore, 
$$
 c_t ( V_+ )
\cdot
 c_t (V)^{-1} = \delta_t (V_+ , V)   + (-1)^{\bv_+ - \bv +1} \left[ t^{-1 } \sum_{p=0}^\infty (-t)^{-p } \sum_{m=0}^{\bv_+} (-1)^{\bv_+ - m} e_{\bv_+ - m}(\nu) h_{p -\bv +1 +m} (\xi)   \right] .
$$
Comparing with \eqref{eqn-quotient-V_-}, we have
$$
 c_t ( V_- )
\cdot
 c_t (V)^{-1} - \delta_t (V_- , V)   = (-1)^{\bv_- - \bv+1} Q \cdot \left(   c_t ( V_+ )
\cdot c_t (V)^{-1} - \delta_t (V_+ , V)   \right).
$$
Multiplying by $c_t(V)$ on both sides proves 1). 
The classical limit can be obtained by setting $Q=0$.

2) can be proved in a completely parallel manner, except that we apply \eqref{eqn rel after anti-sym - neg} instead of ~\eqref{eqn rel after anti-sym}.
\end{proof}

\begin{Remark}
The similar computation already appears in \cite{GK} for the usual quantum cohomology of quiver flag varieties, which we will review in \S \ref{Sec-GW-QH}.
\end{Remark}

\vspace{4ex}

\section{Cluster algebras and quasimap $\QH$}\label{sec:clusterqcoh}

\subsection{Seeds and mutation rules}

We recall the basic definitions of cluster algebras \cite{FZ-1, FZ-2, BFZ-3, FZ-4}. 
We mainly follow the convention in \cite{FZ-4}, adapted to the quiver case. 

 A \emph{semifield} is an abelian multiplicative group, with multiplication $\cdot$, together with an addition $\oplus$, which is commutative, associative, and distributive with respect to $\cdot$. 

As a special example, let $J$ be a finite set, the \emph{tropical semifield} $\bP := \Trop ( v_j , \, j\in J)$ is the free abelian group generated multiplicatively by independent variables $v_j$, $j\in J$, whose addition $\oplus$ is defined as
$$
\prod_{j\in J} v_j^{a_j} \oplus \prod_{j\in J} v_j^{ b_j} := \prod_{j\in J} v_j^{ \min \{ a_j, b_j\}} . 
$$

A \emph{labeled seed} is a triple $\bs = (\bx, \by, B)$, where $\bx = (x_1, \cdots, x_n)$ is an $n$-tuple of independent variables, $\by = (y_1, \cdots, y_n)$ is an $n$-tuple of elements in $\bP$, and $B$ is a skew-symmetric\footnote{In general, $B$ can be allowed to be skew-symmetrizable or even sign-skew-symmetric. See ~\cite{BFZ-3}.} $n\times n$ matrix.

\begin{Definition}
Given a labeled seed $(\bx, \by, B)$, and $k\in \{1, \cdots, n\}$, the \emph{seed mutation} $\mu_k$ in the direction $k$ transforms $(\bx, \by, B)$ into another labeled seed $\mu_k (\bx, \by, B) = (\bx', \by', B')$, defined as follows.
\begin{itemize}
\setlength{\parskip}{1ex}

\item $B' = (b'_{ij})$, where
$$
b'_{ij} = \left\{
\begin{aligned}
& -b_{ij} , && \qquad i = k \text{ or } j = k \\
& b_{ij} + b_{ik} b_{kj} , && \qquad b_{ik}, b_{kj} >0 \\
& b_{ij} - b_{ik} b_{kj} , && \qquad b_{ik}, b_{kj} < 0  \\
& b_{ij} , && \qquad \text{otherwise}
\end{aligned}
\right.
$$

\item The coefficients $\by' = (y'_1, \cdots, y'_n)$ are
$$
y'_j = \left\{
\begin{aligned}
& y_k^{-1} , && \qquad j=k \\
& y_j (1 \oplus y_k^{-1})^{-b_{kj}} , && \qquad b_{kj} >0 \\
& y_j (y_k \oplus 1)^{-b_{kj}} , && \qquad b_{kj} < 0  \\
\end{aligned}
\right.
$$

\item The cluster $\bx' = (x'_1, \cdots, x'_n)$ is given by $x'_j = x_j$ for $j\neq k$, and $x'_k \in \cF$ is determined by the \emph{exchange relation}
\begin{equation} \label{eqn-adj-ex-rel}
x_k x'_k = \frac{y_k}{y_k \oplus 1} \prod_{i: \, b_{ik} >0} x_i^{b_{ik}} + \frac{1}{y_k \oplus 1} \prod_{i: \, b_{ik} <0} x_i^{-b_{ik}} \,.
\end{equation}

\end{itemize}
\end{Definition}

Let $\bZ\bP$ be the group algebra of $\bP$. 
Denote by $\cF = \bZ\bP (\bx):=  \bZ\bP (x_1, \cdots, x_n)$ the field of rational functions in $x_1, \cdots, x_n$ with coefficients in $\bZ\bP$, and by $\bZ\bP [\bx^{\pm 1}] := \bZ\bP [x_1^{\pm 1}, \cdots, x_n^{\pm 1}]$ the ring of Laurent polynomials. 

\subsection{Cluster algebras}

In our case, we will consider cluster algebras arising from quivers. 
Let $\bfQ$ be a quiver as in \S \ref{Sec-quiver}, where we label the gauge nodes as $\bfI = \{1, \cdots, n\}$. 

From now on, we fix an \emph{initial seed} $\bs_0 = (\bx, \by, B)$ as follows.
\begin{itemize}
\setlength{\parskip}{1ex}

\item $
\bx = (x_1, \cdots, x_n)$ are the \emph{initial cluster variables}, and $\by = (y_1, \cdots, y_n )
$ are the \emph{initial coefficients}. 

\item $B$ is  the $n\times n$ adjacency matrix for the gauge nodes
$$
B = ( b_{ij})_{1\leq i, j\leq n}, \qquad b_{ij} :=  \sharp \{ \text{arrows } i \to j \} - \sharp \{ \text{arrows } j \to i \},
$$
which is automatically skew-symmetric. 
\end{itemize}

 Given another labeled seed $\bs = \mu_{k_1} \cdots \mu_{k_l} (\bs_0)$ obtained from $\bs_0$ by a sequence of mutations, the \emph{cluster variables}  of $\bs$ are 
$$
x_{k; \bs} = \mu_{k_1} \cdots \mu_{k_l} (x_k), \qquad 1\leq k\leq n, 
$$
which lie in $\cF$.
Similarly, we have $y_{k; \bs}$, and $B^\bs = (b^\bs_{ij})$. 
Note that we will always reserve the letters $x_k$ for the initial cluster variables, and refer to a general cluster variable as $x_{k; \bs}$. 

\begin{Definition} \label{Defn-cluster}
The \emph{cluster algebra} $\cA$ is the subalgebra in $\cF$, generated by all cluster variables $x_{k; \bs}$, for all $1\leq k\leq n$, and all possible $\bs$ as above.
\end{Definition}

Alternatively, if we label the frozen nodes of $\bfQ$ as $\bfF = \{n+1, \cdots, m\}$, with $m\geq n$, the cluster algebra may be described in terms of the entire $m \times n$ adjacency matrix 
$$
\widetilde B = (b_{ij})_{1\leq i\leq m, \, 1\leq j\leq n}, \qquad b_{ij} :=  \sharp \{ \text{arrows } i \to j \} - \sharp \{ \text{arrows } j \to i \}.
$$
The mutation rule for $\widetilde B$ is the same as $B$.

Let $m:= n + |J|$, and write $x_{n+j} := v_j$. 
The tropical field and the coefficients may be rewritten as 
$$
\bP = \Trop (x_{n+1}, \cdots, x_m), \qquad y_{k; \bs} = \prod_{i=n+1}^m x_i^{b_{ik}^\bs}. 
$$
The mutation rules for cluster variables can then be rewritten as 
\begin{equation} \label{eqn-mutation-2}
x_{k; \bs} x_{k; \bs'} = \prod_{i>n, \, b_{i k}^\bs >0} x_{i}^{b^\bs_{i k}}    \prod_{i\leq n, \, b_{ik}^\bs >0} x_{i; \bs}^{b^\bs_{ik}} + \prod_{i>n, \, b^\bs_{i k} <0} x_{i}^{-b^\bs_{i k}} \prod_{i\leq n, \, b^\bs_{ik} <0} x_{i; \bs}^{-b^\bs_{ik}}. 
\end{equation}

In particular, the initial seed $\bs_0$ thus defined is of \emph{geometric type} and \emph{totally mutable}, in the sense of \cite[\S 1]{BFZ-3}.
In this case, the following result shows that $\cA$ can be characterized only via the \emph{adjacent} mutations. 
Recall that $x'_k := \mu_k (x_k)$. 

\begin{Proposition}[{\cite[Cor. ~1.21]{BFZ-3}}] \label{Prop-gen-by-adj}
Let $\cA$ be the cluster algebra defined above, and suppose that the quiver $\bfQ$ is oriented-acyclic, i.e. has no oriented cycles. 
Then:
\begin{itemize}
\setlength{\parskip}{1ex}
\item $\cA$ is generated by $x_1, x'_1, \cdots, x_n, x'_n$ over $\bZ\bP$.

\item The ideal of relations among the generators $x_1, x'_1, \cdots, x_n, x'_n$ is generated by the exchange relations \eqref{eqn-adj-ex-rel}.
\end{itemize}
\end{Proposition}

One of the remarkable features of cluster algebras is the \emph{Laurent phenomenon}, which guarantees that any cluster variable is actually a Laurent polynomial in the initial cluster variables, rather than simply a rational function.

\begin{Proposition}[Strong Laurent phenomenon {\cite[Prop. ~11.2]{FZ-2}}] \label{Prop-strong-Laurent}
Given the cluster algebra $\cA$ defined above, then for any $\bs$ obtained from $\bs_0$ via mutations, and any $1\leq k\leq n$, 
$$
x_{k; \bs} \in  \bZ [x_1^{\pm 1}, \cdots, x_n^{\pm 1}, x_{n+1}, \cdots, x_{m}].
$$
\end{Proposition}

\subsection{From cluster to $\QH_\sT$}
Let $X$ be the quiver variety as in \S \ref{Sec-quiver}. 

Let $\zeta_1, \cdots, \zeta_{m}$ be another set of independent variables, and consider the map\footnote{Recall that $\bC[Q^{\pm 1}] := \bC[ (Q^{(1)} )^{\pm 1}, \cdots, ( Q^{(n)} )^{\pm 1} ]$, and we denote $\bC[\zeta^{\pm 1} ] := \bC [\zeta_1^{\pm 1}, \cdots, \zeta_{m}^{\pm 1}]$.}
$$
\bC[Q^{\pm 1}] \to \bC[\zeta^{\pm 1} ] , \qquad Q^{(k)} \mapsto (-1)^{\bv_k^- - \bv_k } \prod_{i=1}^{m} \zeta_i^{- b_{ik}} , \qquad 1\leq k\leq n.
$$
We will still denote the images of $Q^{(k)}$ by $Q^{(k)}$.
Let $\bC[Q, \zeta_{\leq n}^{\pm 1}, \zeta_{>n}]$ be the subalgebra in $\bC[\zeta^{\pm 1}]$ generated by $(Q^{(k)} )^{\sign(\theta_k)}$, $1\leq k\leq n$ and $\bC[\zeta_1^{\pm 1} , \cdots, \zeta_n^{\pm 1}, \zeta_{n+1}, \cdots, \zeta_m]$.

\begin{Theorem} \label{Thm-emb}
If the quiver $\bfQ$ is oriented-acyclic, i.e. does not have oriented cycles, then there is an injective map of algebras 
$$
\psi: \cA \hookrightarrow \QH_\sT^\poly (X) [t] \otimes_{\bC[Q]} \bC[ \zeta^{\pm 1} ], 
$$
such that 
$$
\psi (x_i) = \zeta_i \cdot c_t (V_i), \qquad 1\leq i\leq m, 
$$
$$
\psi (x'_k) = \zeta_k^{-1} \prod_{i: \, b_{ik} >0} \zeta_i^{ b_{ik}}  \cdot  \delta_t (V_k^- , V_k ) + \zeta_k^{-1} \prod_{i: \, b_{ik} <0} \zeta_i^{-b_{ik}} \cdot \delta_t (V_k^+ , V_k) , \qquad 1\leq k\leq n.
$$
The images of all cluster variables lie in the subalgebra $\QH_\sT^\poly (X) [t] \otimes_{\bC[Q]} \bC[ Q, \zeta_{\leq n}^{\pm 1}, \zeta_{>n} ]$.
\end{Theorem}

\begin{proof}
By Prop. ~\ref{Prop-gen-by-adj} and the strong Laurent phenomenon (Prop. ~\ref{Prop-strong-Laurent}), to show the existence of $\psi$, it suffices to check that the images of $x_i$, $1\leq i\leq n$, and $x'_k$, $1 \leq k\leq n$, satisfy the exchange relations ~\eqref{eqn-mutation-2}, which for us is
$$
x_k x'_k =  \prod_{i: \, b_{ik} >0} x_i^{b_{ik}} + \prod_{i: \, b_{ik} <0} x_i^{-b_{ik}} , \qquad 1\leq k \leq n.
$$
After the base change to $\bC[\zeta^{\pm 1}]$, the exchange relations in Thm. ~\ref{Thm-c_t-rel} can be written in a unified form:
\begin{eqnarray}
&& c_t (V_k) \cdot \Big[  \delta_t (V_k^- , V_k ) + (-1)^{\bv_k^- - \bv_k }  Q^{(k)}  \cdot \delta_t (V_k^+ , V_k) \Big]  \nonumber \\
&=&  \prod_{e\in \bfQ_1, \, t(e) = k} c_t (V_{s(e)}) + (-1)^{\bv_k^- - \bv_k }  Q^{(k)}  \cdot \prod_{e\in \bfQ_1, \, s(e) = k} c_t (V_{t(e)})  \nonumber \\
&=&  \prod_{ i: \, b_{ik} >0 } c_t (V_i)^{b_{ik}} + \prod_{i=1}^{m} \zeta_i^{- b_{ik}} \cdot  \prod_{i: \, b_{ik}< 0} c_t (V_i)^{ - b_{ik}}   \label{eqn-ex-proof}
\end{eqnarray}
in $\QH_\sT (X)[t]$. 
The result follows from a redistribution of $\zeta_i$'s.
The last statement follows from Prop. ~\ref{Prop-strong-Laurent}.

It remains to prove the injectivity. 
Suppose that $f (x_1, \cdots, x_n) \in \ker \psi$, i.e. $f$ is a rational function in $x_k$'s, with coefficients in $\bZ\bP$, such that
$$
f( \zeta_1 \cdot c_t (V_1), \cdots, \zeta_n \cdot c_t (V_n)) = 0. 
$$
Possibly replacing $f$ by its numerator, we may assume that $f$ is a polynomial.
Now in $\QH_\sT^\poly (X) [t]$ we know that
$$
c_t (V_k) = c_t (\cV_k) + O(Q).
$$
Here $\cV_k := V_k |_X$ is the (classical) tautological bundle. 
A further expansion in $t$ shows that
$$
f(\zeta_1 t^{\bv_1} \cdot (1 + O(Q, t^{-1})) , \cdots , \zeta_n t^{\bv_n} \cdot (1 + O(Q, t^{-1} )) ) = 0, 
$$
where each $O(Q, t^{-1})$ stands for a polynomial in $Q$ and $t^{-1}$ divisible by $Q t^{-1}$.
Since $\zeta_1, \cdots, \zeta_n$ are algebraically independent, we know that $\zeta_1 t^{\bv_1} \cdot (1 + O(Q, t^{-1})) , \cdots , \zeta_n t^{\bv_n} \cdot (1 + O(Q, t^{-1} ))$ are also algebraically independent.
This implies  $f=0$, and hence the injectivity.
\end{proof}

\begin{Remark} \label{Rk-inv}
We would like to ask whether it is possible to invert $\psi(x_k)$. 
The only subtle obstruction is that the inverse may no longer be a polynomial in $Q$, and hence the base change to $\zeta$-variables may be pathological. 
Therefore, we may consider the case where $\widetilde B$ is \emph{of full rank}. 
Then, the map $\bC[Q] \to \bC[\zeta^{\pm 1}]$ is injective, and it makes sense to define
$$
\bC[\zeta^{\pm 1}, Q]\!] := \text{completion of } \bC[\zeta^{\pm 1}] \text{ at } \fm_Q. 
$$
Now for a class $c_t (V_k)$, its inverse is well-defined in
$$
c_t (V_k)^{-1}  \in \QH_\sT (X)[t]_\loc :=  \QH_\sT (X) \otimes_{H_T^*(\pt)} \Frac(H_T^*(\pt) [t]).
$$
Note that $\QH_\sT (X)$ already contains the power series ring $\bC[\![Q]\!]$. 
The element $\psi(x_k^{-1})$ is then well-defined if we slightly enlarge the target as
$$
\psi (x_k^{-1}) \in \QH_\sT (X)[t]_\loc \otimes_{\bC[\![Q]\!]} \bC[\zeta^{\pm 1}, Q]\!].
$$
With this modification, in the same manner as Thm. ~\ref{Thm-emb}, one may define a map from the \emph{upper cluster algebra} $\cA^\up$ to $\QH_\sT (X)[t]_\loc \otimes_{\bC[\![Q]\!]} \bC[\zeta^{\pm 1}, Q]\!]$.
\end{Remark}

\subsection{$\bg$-vectors}

In this subsection, we assume that the initial seed $\bs_0$ is of \emph{principal coefficients}, i.e. $m=2n$, and the matrix $\widetilde B = \begin{pmatrix}
B \\
I_n
\end{pmatrix}$.
The initial coefficients are $(y_1, \cdots, y_n) = (v_1, \cdots, v_n) = (x_{n+1} , \cdots, x_{2n})$.

Furthermore, we take the \emph{negative} generic stability condition $\theta$, i.e. satisfying $\theta_k <0$, for all $1\leq k\leq n$. 
In particular,  
$$
\bC[Q^{\Eff(X)}] = \bC[(Q^{(1)} )^{-1}, \cdots, (Q^{(n)})^{-1} ] , \quad \text{with images} \quad ( Q^{(k)} )^{-1} \mapsto (-1)^{\bv_k^- - \bv_k } \zeta_{n+k} \prod_{i=1}^{n} \zeta_i^{ b_{ik}}.
$$

According to \cite[Prop. ~6.1]{FZ-4}, the cluster algebra $\cA$ is endowed with a $\bZ^n$-grading, if we set
$$
\deg x_k = \be_k, \qquad \deg y_k = - \bb_k := - \sum_{i=1}^n b_{ik} \be_i, \qquad 1\leq k \leq n.
$$
where $\be_i \in \bZ^n$ is the $i$-th standard basis vector, and $\bb_k$ is the $k$-th column of the matrix $B$.
Let $\bs$ be a labeled seed obtained from the initial seed $\bs_0$ by a sequence of mutations. 
The \emph{$\bg$-vectors} are defined as 
$$
\bg_{k; \bs} := \deg x_{k; \bs} \quad \in \quad \bZ^n, \qquad 1\leq k\leq n. 
$$ 
According to \cite[Prop. ~6.6]{FZ-4}, they can be determined inductively by mutations.

Moreover, there is another set of variables 
$$
\hat y_k := y_k \prod_{i=1}^n x_i^{b_{ik}},  \qquad \text{satisfying } \deg \hat y_k = 0 , \qquad 1\leq k\leq n. 
$$
Since $\widetilde B$ is of full rank, by Rem. ~\ref{Rk-inv}, the images $\psi(\hat y_k)$ are well-defined if we slightly enlarge the target 
$$
\psi (\hat y_k) \in \QH_\sT (X)[t]_\loc \otimes_{\bC[\![Q^{\Eff(X) }]\!]} \bC[\zeta^{\pm 1}, Q^{\Eff(X) }]\!],
$$
where $\bC[\zeta^{\pm 1} , Q^{\Eff(X)}]\!]$ is the completion of $\bC[\zeta^{\pm 1}]$ with respect to $\fm_Q = ((Q^{(1)} )^{-1}, \cdots, (Q^{(n)})^{-1} )$.


\begin{Proposition} \label{Prop-g-vector}
Let the quiver $\bfQ$ be oriented-acyclic, such that $\widetilde B$ is of full rank.
For any $\bs$ as above, in the enlarged ring $ \QH_\sT (X) [t]_\loc \otimes_{\bC[\![Q^{\Eff(X)}]\!]}  \bC[\zeta^{\pm 1}, Q^{\Eff(X) }]\!]$, 
$$
\psi (x_{k; \bs}) = \prod_{i=1}^n (\zeta_i \cdot c_t (V_i) )^{g_i} \cdot (1+ O(Q^{\Eff(X)})) ,\qquad 1\leq k\leq n,
$$
where  $g_{k; \bs} = (g_1, \cdots, g_n) \in \bZ^n$, and $O(Q^{\Eff(X)})$ stands for an element in $\fm_Q \cdot \QH_\sT (X) [t]_\loc$. 
\end{Proposition}

\begin{proof}
By \cite[Cor. ~6.3]{FZ-4}, we have
$$
x_{k; \bs} = 
\frac{ F_{k; \bs} |_\cF (\hat y_1, \cdots, \hat y_n)
}{ 
F_{k; \bs} |_\bP (y_1, \cdots, y_n) 
} 
x_1^{g_1} \cdots x_n^{g_n}, 
$$
where $F_{k; \bs} (y_1, \cdots, y_n) \in \bZ[y_1, \cdots, y_n]$ is the $F$-polynomial as in \cite[Def. ~3.3]{FZ-4}, and $F_{k; \bs} |_\cF$, $F_{k; \bs} |_\bP$ are its evaluations in the corresponding fields of coefficients. 
In particular, for principal coefficients, we have $F_{k; \bs} |_\bP (y_1, \cdots, y_n)$. 
Moreover, \cite[Conj. ~5.4]{FZ-4}, which is proved in \cite[Thm. ~1.7]{DWZ-2}, shows that $F_{k; \bs}$ has constant term $1$. 

Now, in our case, viewed in $ \QH_\sT (X) [t]_\loc \otimes_{\bC[\![Q^{\Eff(X)}]\!]}  \bC[\zeta^{\pm 1}, Q^{\Eff(X) }]\!]$, we have
$$
\psi (\hat y_k) = \zeta_{n+k} \prod_{i=1}^n \zeta_i^{b_{ik}} \cdot c_t (V_{n+k}) \prod_{i=1}^n c_t (V_i)^{b_{ik}} \quad  \in \quad ( Q^{(k)} )^{-1} \cdot \QH_\sT (X)[t]_\loc.
$$
The result follows.
\end{proof}

\subsection{Example: type $A_n$} \label{Sec-A_n}

Consider the $A_n$ quiver with one frozen node
\[
\xymatrix{
[\bv_0] \ar[r] & (\bv_1) \ar[r] & (\bv_2) \ar[r] & \cdots \ar[r] & (\bv_n) \\
}
\]
where the numbers are the dimensions, $(-)$ stands for gauge nodes, and $[-]$ stands for frozen nodes. 
We choose the stability condition and dimension vector as follows:
\begin{itemize}
\setlength{\parskip}{1ex}

\item $\theta_k>0$, for any $1\leq k\leq n$; 

\item $\bv_k > \bv_{k+1}$, for any $0\leq k\leq n-1$;

\item $\bv_0 - \bv_k < \bv_{k+1} - \bv_{k+2}$, for any $0 \leq k\leq n-1$.

\end{itemize}
Here we set $\bv_k = 0$ for $k>n$. 

The quiver variety $X$ is the partial flag variety $\mathrm{Fl} (\bv_n, \cdots, \bv_1, \bv_0)$. 
Note that since $\bv_k \geq \bv_{k+1}$, $V_k^- = V_{k-1}$, $V_k^+ = V_{k+1}$, and $\delta_t (V_{k+1}, V_k) = 0$.
Thm. ~\ref{Thm-c_t-rel} now reduces to
\begin{equation} \label{eqn-mutation-A-type}
c_t (V_k)   \cdot \delta_t (V_{k-1} , V_k)  =  c_t ( V_{k-1} )  + (-1)^{\bv_{k-1} - \bv_k  } Q^{(k)} \cdot   c_t ( V_{k+1} )
 . 
\end{equation}

\begin{Lemma} \label{Lemma-other-rel-A}
For $1\leq k\leq l \leq n$, we have
\begin{enumerate}[label=\arabic*)]
\setlength{\parskip}{1ex}
\item $\displaystyle
c_t (V_{l} ) \cdot \delta_t (V_{k-1} , V_{l} )  = c_t (V_{k-1} ) + (-1)^{\bv_{l-1} - \bv_l  } Q^{(l)}  \cdot \delta_t (V_{k-1} , V_{l-1} ) \cdot c_t (V_{l+1})$,

\item $\displaystyle 
\delta_t (V_{k-1} , V_{l} )  \cdot \delta_t (V_l, V_{l+1} )  = \delta_t (V_{k-1}, V_{l+1} ) + (-1)^{\bv_{l-1} - \bv_l } Q^{(l)}  \cdot \delta_t (V_{k-1} , V_{l-1} )
$ ,
\end{enumerate}
where we set $V_{n+1} = 0$.
\end{Lemma}

\begin{proof}
We prove by induction on $l-k$. 
For $l=k$, 1) is \eqref{eqn-mutation-A-type}. 
Now let us assume 1) for $l$ and prove 1) for $l+1$ and 2) for $l$. 
Divide both sides of the equation for $l$ by $c_t (V_{l+1})$. 
\begin{equation} \label{eqn-divide}
\frac{c_t (V_{l} )}{c_t (V_{l+1})} \cdot \delta_t (V_{k-1} , V_{l} )  = \frac{ c_t (V_{k-1} ) }{c_t (V_{l+1} ) } + (-1)^{\bv_{l-1} - \bv_l  } Q^{(l)}  \cdot \delta_t (V_{k-1} , V_{l-1} ) ,
\end{equation}
whose LHS is 
$$
\delta_t (V_l , V_{l+1}) \cdot \delta_t (V_{k-1} , V_{l} )  + \Big[ \frac{c_t (V_{l} )}{c_t (V_{l+1})} \Big]_- \cdot \delta_t (V_{k-1} , V_{l} ) .
$$
Thm. ~\ref{Thm-c_t-rel} now implies
$$
\Big[ \frac{c_t (V_{l} )}{c_t (V_{l+1})} \Big]_- \cdot \delta_t (V_{k-1} , V_{l} ) =  - (-1)^{\bv_l - \bv_{l+1} } Q^{(l+1)} \cdot \Big[ \dfrac{c_t (V_{l+2})}{c_t (V_{l+1})} \Big]_- \cdot \delta_t (V_{k-1} , V_{l} ) 
$$
Since $\bv_{l+2} - \bv_{l+1} + \bv_{k-1} - \bv_l \leq \bv_{l+2} - \bv_{l+1} + \bv_{0} - \bv_l  <0$, the above contains only negative powers of $t$. 
Therefore, the equation \eqref{eqn-divide} splits into positive and negative parts:
$$
\delta_t (V_l, V_{l+1} ) \cdot \delta_t (V_{k-1} , V_{l} )  = \delta_t (V_{k-1}, V_{l+1} ) + (-1)^{\bv_{l-1} - \bv_l } Q^{(l)}  \cdot \delta_t (V_{k-1} , V_{l-1} ) ,
$$
$$
 - (-1)^{\bv_l - \bv_{l+1} } Q^{(l+1)} \cdot \Big[ \dfrac{c_t (V_{l+2})}{c_t (V_{l+1})} \Big]_- \cdot \delta_t (V_{k-1} , V_{l} ) = \Big[ \frac{ c_t (V_{k-1} ) }{c_t (V_{l+1} ) } \Big]_-. 
$$
The positive part is exactly 2) for $l$. 
In view of $\delta_t (V_{l+2}, V_{l+1}) = 0$, the negative part is
$$
- (-1)^{\bv_l -\bv_{l+1}} Q^{(l+1)} \cdot  \dfrac{c_t (V_{l+2})}{c_t (V_{l+1})}  \cdot \delta_t (V_{k-1} , V_{l} ) = \frac{ c_t (V_{k-1} ) }{c_t (V_{l+1} ) } - \delta_t (V_{k-1}, V_{l+1} ). 
$$
which gives 1) for $l+1$. 
\end{proof}

Let us consider the cluster variables. 
The cluster algebra $\cA$ is the one of type $A_n$, with one coefficient coming from the frozen node. 
Various explicit models for this cluster algebra are well-known, e.g. in terms of triangulations of an $(n+3)$-gon (\cite[Lemma ~5.3.1]{FWZ-45}), or minors in the reduced double Bruhat cell \cite[Example ~5.3.9]{FWZ-45}. 

In particular, the \emph{non-initial} cluster variables are parameterized by positive roots 
$$
\alpha_{kl} := \alpha_{k+1} + \cdots + \alpha_{l}, \qquad 0\leq k \leq l \leq n. 
$$
where $\alpha_i$'s are the simple roots for $A_n$. 
We will write the corresponding cluster variable by $x[ {\alpha_{kl}}]$, or simply by $x[kl]$. 
Note that $x[\alpha_k] = x[k-1, k]$. 
The initial cluster variables are written as $x_k = x[-\alpha_k]$.

\begin{Proposition} \label{Prop-cl-var-A}
Under the map $\psi$, images of all the cluster variables are
$$
\psi(x[-\alpha_k]) = \zeta_k \cdot c_t (V_k), \qquad \psi ( x[kl] ) = \zeta_k \zeta_l^{-1} \cdot \delta_t (V_k, V_l), \qquad 0\leq k \leq l \leq n. 
$$
\end{Proposition}

\begin{proof}
$\psi(x[-\alpha_k]) = \zeta_k \cdot c_t (V_k)$ and $\psi(x[\alpha_k]) =  \zeta_{k-1} \zeta_{k}^{-1} \cdot \delta_t (V_{k-1}, V_{k})$ simply follows from the definition of $\psi$, and the fact that $\delta_t (V_{k+1}, V_{k}) = 0$. 

For the non-initial cluster variables, since\footnote{The $\zeta$-variables are $\zeta_0, \zeta_1, \cdots, \zeta_n$, and we set $\zeta_{n+1} = 1$.} $(-1)^{\bv_{l-1} - \bv_l  } Q^{(l)} = \zeta_{l-1}^{-1} \zeta_{l+1}$ for $l\geq 1$, Lemma \ref{Lemma-other-rel-A} 2) shows that
$$
 \zeta_{k-1} \zeta_{l}^{-1}  \cdot \delta_t (V_{k-1} , V_{l} )  \cdot  \zeta_{l} \zeta_{l+1}^{-1}  \cdot \delta_t (V_l, V_{l+1} )  =  \zeta_{k-1} \zeta_{l+1}^{-1}  \cdot \delta_t (V_{k-1}, V_{l+1} ) + \zeta_{k-1} \zeta_{l-1}^{-1}  \cdot \delta_t (V_{k-1} , V_{l-1} ), 
$$
for $1\leq k\leq l \leq n$.
The statement then follows from induction on $l-k$, and a comparison with the standard relations among cluster variables of type $A_n$:
$$
x[k-1, l] \cdot x[l, l+1] = x[k-1, l+1] + x[k-1, l-1], 
$$
which can be found in \cite[Exercise ~5.3.10]{FWZ-45} or ~\cite[\S 12.2]{FZ-2} (with a different triangulation, as in ~\cite[~\S 5.3]{FWZ-45}).
\end{proof}


\vspace{4ex}

\section{Quiver flag: comparison with QH} \label{Sec-GW-QH}

In this section, we consider a special type of quiver varieties, and show that our results in the previous sections can be applied to the usual quantum cohomology, defined in terms of genus-zero Gromov--Witten invariants.

\subsection{Quiver flag varieties}

Let $\bfQ$ be a quiver as in \S \ref{Sec-quiver}, with gauge nodes labeled as $\bfI = \{1, \cdots, n\}$. 

\begin{Assumption} \label{Ass-quiver-flag}
We further assume that:
\begin{itemize}
\setlength{\parskip}{1ex}

\item $\bfQ$ is oriented-acyclic.

\item There is only one frozen node $\bfF = \{0\}$, which is the only source of the quiver.

\item The stability condition $\theta$ is chosen such that $\theta_k >0$, for all $k\in \bfI$. 

\item $\bv_k^- - \bv_k^+ \geq 2$, where $\bv_k^- := \sum_{e\in \bfQ_1, \, t(e) = k} \bv_{s(e)}$ and  $\bv_k^+ := \sum_{e\in \bfQ_1, \, s(e) = k} \bv_{t(e)}$, for all $k\in \bfI$.

\end{itemize}
\end{Assumption}

Let $X$ be the associated quiver variety. 
The first three assumptions mean that $X$ is a \emph{quiver flag variety} in the sense of ~\cite{Cra}.
In particular, by Lemma ~\ref{Lemma-stab}, for a representative $(B_e, V_i)_{e\in \bfQ_1, i\in \bfQ_0}$ of a point in $X$, all maps 
$$
\sum_{e \in \bfQ_1,  \, t(e) = k  } B_e: \  V_k^- := \bigoplus_{e \in \bfQ_1, \, t(e) = k} V_{s(e)} \to V_k, \qquad 1\leq k\leq n
$$ 
are surjective. 
In particular, we have $\bv_k^-  \geq \bv_k$. 

\begin{Remark}
By ~\cite[Cor. ~3.8]{Cra}, $X$ is Fano if $\bv_k^- > \bv_k^+$. 
Our last assumption is stronger than that, which is an analogue of a ``Fano index $\geq 2$" condition.
In this case, the comparison between quasimap QH and the usual one can be obtained from the general $\epsilon$-wall-crossing technique \cite{CK-wall, CK-higher-g}. 
In the current paper, we compare them by the explicit computations, while reserving this general $\epsilon$-wall-crossing treatment for future works. 
\end{Remark}

Consider the insertion
$$
\sigma^{(k)}_l (\xi) := e_l (\xi^{(k)}_1, \cdots, \xi^{(k)}_{\bv_k} ) \quad    \in H_\sT^* (\pt) \otimes_\bC H^*_\sK(\pt)^\sW, \qquad 1\leq k\leq n, \ 0\leq l\leq \bv_k.
$$
The tautological classes $\sigma^{(k)}_l (\xi) |_X$ generate the cohomology $H^*_\sT (X)$. 

\begin{Proposition} \label{Prop-rigidity-quiver-flag}
Let $X$ be a quiver flag variety satisfying Ass. ~\ref{Ass-quiver-flag}. 
Then
$$
\widehat I^{( \sigma_l^{(k)} (\xi) )} (Q; \bh) = \widehat{\sigma_l^{(k)} (\xi)} (Q) = \sigma_l^{(k)} (\xi) |_X .
$$
\end{Proposition}

\begin{proof}
According to \cite[Prop. ~2.4.8,  Cor. ~2.4.13]{Kal}, $\Eff (X^\ab)$ is generated by $e_i^{(k)} \in \sX_* (\sK)$, for $1\leq k\leq n$, $1\leq i\leq \bv_k$. 
We compute
\ben
\la \det \sN , e_1^{(k)} \ra &=& \Big\la \sum_{e\in \bfQ_1} \sum_{a=1}^{\bv_{s(e)}} \sum_{b =1}^{\bv_{t(e)}} ( \xi^{(t(e))}_b - \xi^{(s(e))}_a ) , e_1^{(k)} \Big\ra \\
&=& \sum_{e\in \bfQ_1, \, t(e) = k} \bv_{s(e)} - \sum_{e\in \bfQ_1, \, s(e) = k} \bv_{t(e)} \\
&=& \bv_k^- - \bv_k^+. 
\een
On the other hand, $\deg_{e_1^{(k)}} \sigma_l^{(k')} (\xi) = \delta_{kk'}$, for all $l$ and $k'$.
By $S_{\bv_k}$-symmetry, the same results hold for all $e_i^{(k)}$'s.
The statement then follows from Prop.  ~\ref{Prop-rigidity}.
\end{proof}

\subsection{Comparison with quantum cohomology}

Let $X$ be a smooth quasi-projective variety, and $\beta \in H_2 (X, \bZ)$ be a curve class.
Let $\sT$ be a torus acting on $X$.
Consider the moduli space ~\cite{Kon, BM}
$$
\overline\cM_{0, n} (X, \beta) := \{ f: (C; p_1, \cdots, p_n) \to X \mid g(C) = 0, \, f_* [C] = \beta \}, 
$$
where $f$ is a stable map from an $n$-pointed nodal prestable curve to $X$.
The moduli space $\overline\cM_{0, n} (X, \beta)$ is a separated Deligne--Mumford stack, with proper evaluation maps $\ev_i: \overline\cM_{0, n} (X, \beta) \to X$. 
By standard construction \cite{BF, Beh}, it is equipped with a perfect obstruction theory and hence admits a virtual fundamental class $[\overline\cM_{0, n} (X, \beta)]_\vir$.

\begin{Definition}
Let $\gamma, \eta \in H^*_\sT (X)$. 
The \emph{small quantum product} of $\gamma$ and $\eta$ is defined as
$$
\gamma *_\GW \eta 
:= \gamma \cdot \eta + \sum_{\beta \in H_2(X, \bZ)_{\Eff}}  Q^\beta \ev_{1*} \big( \ev_2^* \gamma \cdot \ev_3^* \eta \cdot [\overline\cM_{0, 3} (X, \beta)]_\vir \big) \quad \in H_\sT ^*(X) [\![Q]\!] .
$$
The \emph{small quantum cohomology} $\QH^\GW_\sT (X)$ of $X$ is the ring $H^*_\sT (X) [\![Q]\!]$, equipped with the  small quantum product.
It is a deformation of the usual cohomology ring, whose identity coincides with the usual one $1$. 
\end{Definition}

When $X$ is Fano, the quantum product has always finite terms, and one may define the quantum cohomology over $\bC[Q]$. 
We will denote this version by $\QH_\sT^{\GW, \poly} (X)$.

Now let $X$ be a quiver flag variety satisfying Ass. ~\ref{Ass-quiver-flag}. 
The quantum cohomology ring is computed in \cite{GK}. 

\begin{Proposition}[{\cite[Thm. ~0.2]{GK}}]
$\QH^\GW_\sT (X)$ has the same explicit presentation as in Thm. ~\ref{Thm-QH}, where a symmetric function $\tau(\xi) \in H_\sT^* (\pt) \otimes_\bC H_{ \sK }^*(\pt)^\sW$ stands for the tautological class $\tau(\xi)|_X$.
\end{Proposition}

Combined with Prop. ~\ref{Prop-rigidity-quiver-flag}, we then obtain a comparison result.

\begin{Corollary} \label{Cor-isom-QH-GW}
Let $X$ be a quiver flag variety satisfying Ass. ~\ref{Ass-quiver-flag}. 
There is an isomorphism of $H_\sT^*(\pt) [\![Q]\!]$-algebras 
$$
\QH_\sT (X) \cong \QH_\sT^\GW (X), 
$$
where $\sigma_l^{(k)} (\xi)$ is sent to $\sigma_l^{(k)} (\xi) |_X$. 
Similar holds for $
\QH_\sT^\poly (X) \cong \QH_\sT^{\GW, \poly } (X).
$
\end{Corollary}

We also obtain a cluster algebra structure on $\QH^{\GW, \poly}_\sT (X) [t] \otimes_{\bC[Q]} \bC[ \zeta^{\pm 1} ]$. 

\begin{Corollary} \label{Cor-emb-GW}
Let $X$ be a quiver flag variety satisfying Ass. ~\ref{Ass-quiver-flag}.
Thm. ~\ref{Thm-emb} holds if we replace $\QH_\sT^\poly (X)$ with $\QH^{\GW, \poly}_\sT (X) $.
\end{Corollary}

\subsection{Example: type $A_n$}

Let $X$ be a type $A_n$ quiver variety as in \S \ref{Sec-A_n}. 
Now we can say more about the truncated Chern quotients.

\begin{Lemma}
For any $0\leq k \leq l \leq n$, 
$$
\widehat I^{( \delta_t (V_k, V_l) )} (Q; \bh) = \widehat{ \delta_t (V_k, V_l) } (Q) = c_t (\cV_k / \cV_l) .
$$
\end{Lemma}

\begin{proof}
By definition, we see that
$$
\deg_{e_1^{(k)}} \delta_t (V_k, V_l) = 1, \qquad \deg_{e_1^{(l)}} \delta_t (V_k, V_l) = \bv_k - \bv_l, 
$$
and $\deg_{e_1^{(k')}} \delta_t (V_k, V_l) = 0$, for $k' \neq k, l$. 
As in Prop. ~\ref{Prop-rigidity-quiver-flag}, $\la \det\sN, e_1^{(k')} \ra = \bv_{k'-1} - \bv_{k'+1}$, which is always $\geq 2$. 
Moreover, by assumption on $\bfQ$, we have $\bv_0 - \bv_{l-1} < \bv_l - \bv_{l+1}$. 
Hence
$$
 \bv_k - \bv_l < \bv_k + \bv_{l-1} - \bv_{l+1} - \bv_0 < \bv_{l-1} - \bv_{l+1}, 
$$
which implies the result.
\end{proof}

In particular, we recover the cluster algebra conjecture for type $A_n$. 

\begin{Corollary}[{\cite[Thm. ~1.3]{HZ}}]
For a quiver of type $A_n$ satisfying the assumptions in \S \ref{Sec-A_n}, there is an injective map $\psi$ from the cluster algebra $\cA$ to $\QH^{\GW, \poly}_\sT (X) [t] \otimes_{\bC[Q]} \bC[ \zeta^{\pm 1} ]$, such that the images of cluster variables are
$$
\psi(x[-\alpha_k]) = \zeta_k \cdot c_t (\cV_k), \qquad \psi ( x[kl] ) = \zeta_k \zeta_l^{-1} \cdot c_t (\cV_k/  \cV_l), \qquad 0\leq k \leq l \leq n. 
$$
\end{Corollary}

\vspace{4ex}

\bibliographystyle{abbrv}
\bibliography{reference}

\end{document}